\documentclass[12pt,reqno]{amsart}
\usepackage{latexsym,amsmath,amsfonts,amscd,amssymb}
\usepackage{graphics}

\usepackage{tikz-cd}

\newcommand{\Sympl}{\operatorname{Sympl}}

\newcommand{\UU}{\operatorname{U}}
\newcommand{\SO}{\operatorname{SO}}
\newcommand{\Sp}{\operatorname{Sp}}
\newcommand{\OO}{\operatorname{O}}
\newcommand{\GL}{\operatorname{GL}}

\DeclareMathOperator{\Id}{Id}
\DeclareMathOperator{\Diff}{Diff}

\DeclareMathOperator{\End}{End}

\newcommand{\la}{\langle}
\newcommand{\ra}{\rangle}

\newcommand{\inc}{\hookrightarrow}
\newcommand{\bd}{\partial}
\newcommand{\x}{\times}

\newcommand{\CP}{{\mathbb C \mathbb P}}

\newcommand{\Aut}{\operatorname{Aut}}

\newcommand{\PD}{\operatorname{PD}}

\newcommand{\ii}{\mathrm{i}}

\newcommand{\cHH}{{\mathcal H}}

\newcommand{\cU}{{\mathcal U}}
\newcommand{\cV}{{\mathcal V}}

\newcommand{\CC}{{\mathbb C}}

\newcommand{\NN}{{\mathbb N}}

\newcommand{\RR}{{\mathbb R}}

\newcommand{\ZZ}{{\mathbb Z}}
\renewcommand{\a}{\alpha}
\renewcommand{\b}{\beta}
\renewcommand{\d}{\delta}
\newcommand{\g}{\gamma}
\newcommand{\e}{\varepsilon}

\renewcommand{\l}{\lambda}

\newcommand{\n}{\nu}
\renewcommand{\o}{\omega}

\newcommand{\G}{\Gamma}
\renewcommand{\O}{\Omega}
\renewcommand{\S}{\Sigma}

\newtheorem{theorem}{Theorem}
\newtheorem{proposition}[theorem]{Proposition}
\newtheorem{lemma}[theorem]{Lemma}
\newtheorem{definition}[theorem]{Definition}

\newtheorem{corollary}[theorem]{Corollary}
\newtheorem{remark}[theorem]{Remark}

\title{Symplectic resolution of orbifolds with homogeneous isotropy}

\author[V. Mu\~{n}oz]{Vicente Mu\~{n}oz}
\address{Facultad de Ciencias Matem\'aticas, Universidad
Complutense de Madrid, Plaza de Ciencias 3, 28040 Madrid, Spain}
\email{vicente.munoz@mat.ucm.es}

\author[J.A. Rojo]{Juan Angel  Rojo}
\address{Facultad de Ciencias Matem\'aticas, Universidad
Complutense de Madrid, Plaza de Ciencias 3, 28040 Madrid, Spain}
\address{Instituto de Ciencias Matem\'aticas (CSIC-UAM-UC3M-UCM),
C/ Nicol\'as Cabrera 15, 28049 Madrid, Spain}
\email{jarcarulli@ucm.es}


\begin{document}

\begin{abstract}
We construct the symplectic resolution of a symplectic orbifold whose isotropy locus consists
of disjoint submanifolds with homogeneous isotropy, that is, all its points have the same
isotropy groups.
\end{abstract}

\maketitle

\section{Introduction}\label{sec:1}

An orbifold is a space which is locally modelled on balls of $\RR^n$ quotient by a finite group.
These have been very useful in many geometrical contexts \cite{Th}. In the setting of 
symplectic geometry, symplectic orbifolds have been introduced mainly as a way to
construct symplectic manifolds by resolving their singularities.
The problem of resolution of singularities and blow-up in the symplectic setting 
was posed by Gromov in \cite{G}. Few years later, the symplectic blow-up was
rigorously defined by McDuff \cite{McD} and it was used to construct a simply-connected
symplectic manifold with no K\"ahler structure.

McCarthy and Wolfson developed in \cite{MW} a symplectic resolution for isolated 
singularities of orbifolds in dimension $4$.
Later on, Cavalcanti, Fern\'andez and the first author gave a method of performing 
symplectic resolution of orbifold isolated singularities in all dimensions \cite{CFM}. This was
used in \cite{FM} to give the first example of a simply-connected symplectic $8$-manifold
which is non-formal, as the resolution of a suitable symplectic $8$-orbifold. This manifold
was proved to have also a complex structure in \cite{BM}.

Niederkr\"uger and Pasquotto \cite{NP1,NP2} provided a method for resolving symplectic
orbifold singularities via symplectic reduction, which can be used for some classes of
symplectic singularities, including cyclic orbifold singularities, even if these are not isolated.
Recently, Chen \cite{C} has detailed a method of resolving arbitrary symplectic $4$-orbifolds,
using the fact that the singular points of the underlying space have to be isolated in dimension $4$.
The novelty is that there can be also surfaces of non-trivial isotropy, 
and the symplectic orbifold form has to be modified on these surfaces also. In this dimension,
the work of the authors with Tralle \cite{MRT} also serve to resolve symplectic $4$-orbifolds
whose isotropy set is of codimension $2$. In such case the orbifold is topologically a manifold
(the isotropy points are non-singular),
so the question only amounts to change the orbifold symplectic form into a smooth symplectic form.

Bazzoni, Fern\'andez and the first author \cite{BFM} have given the first construction of a 
symplectic resolution of an orbifold of dimension $6$ with isotropy sets of dimension $0$ and
$2$, although the construction is ad hoc for the particular example at hand as it satisfies that
the normal bundle to the $2$-dimensional isotropy set is trivial. This was used to give
the first example of a simply-connected non-K\"ahler manifold which is simultaneously complex
and symplectic.
symplectic tubular
In this paper we give a procedure to resolve a wider type of singularities in a symplectic 
orbifold $X$ of arbitrary dimension $2n$. 
We are able to develop such resolution for orbifolds $X$ whose isotropy set is composed of disjoint submanifolds $D_i$
so that each of the $D_i$ have the same isotropy groups at all its points. We call this $D_i$ a homogeneous
isotropy set and such orbifold $X$ a HI orbifold.
This allows the existence of positive dimensional submanifolds composed of singular points.
The singular points of the topological underlying space are not isolated, hence new 
techniques are required in order to perform the resolution. We are able to endow the 
normal bundle to $D_i$ with a nice structure in which to
effectively perform fiberwise the algebraic resolution of singularities 
of \cite{EV}, and then glue these local resolutions into a resolution $\tilde{X}$ of $X$.

The general strategy is to endow the normal bundle $\nu_D$ of any homogeneous 
isotropy submanifold $D \subset X$ with the structure of an orbifold bundle with structure group $\UU(k)$,
where $2k$ is the codimension of $D$. 
The singularities of $X$ at the points of $D$ are quotient singularities in the fibers $F=\CC^k/\G$ of $\nu_D$,
where $\G$ is the isotropy group of $D$.
The usual resolution of singularities for algebraic geometry allows to resolve each of the fibers $F$ of $\nu_D$
separately. However, we need this resolution to glue nicely when we change trivializations.
For this we need an improvement of the classical theorem of resolution of singularities by Hironaka \cite{H}.
This improvement is the \emph{constructive resolution of singularities}
developed by Encinas and Villamayor \cite{EV},
which is compatible with group actions. 
Using their result we are able to construct the resolution $\tilde\nu_D$ of $X$ near $D$ as a smooth manifold.

The resolution $\tilde\nu_D$ has the structure of a fiber bundle over $D$, 
with fiber the resolution $\tilde{F}$ of $F= \CC^k/\G$.
Both base $D$ and fiber $\tilde{F}$ of the total space $\tilde\nu_D$ are symplectic, 
but this does not imply directly that $\tilde\nu_D$ admits a 
symplectic form. First we need to prove that there is no cohomological obstruction for this, which amounts
to finding a cohomology class on the total space $\tilde\nu_D$ that restricts to the cohomology class of 
the symplectic form of the fiber. Secondly, we have to develop a globalization procedure for symplectic
fiber bundles with non-compact symplectic fiber. The final step is 
to glue the symplectic form on $\tilde\nu_D$ with the original symplectic form of $X  -  D$.

The main result is:

\begin{theorem} \label{thm:main-thm}
Let $(X, \o)$ be a symplectic orbifold with isotropy set consisting of disjoint homogeneous isotropy 
subsets. Then there exists a symplectic manifold $(\tilde{X}, \tilde{\o})$ and a smooth map 
$b: (\tilde{X},\tilde{\o}) \to (X,\o)$
which is a symplectomorphism outside an arbitrarily small neighborhood of the isotropy set of $X$.
\end{theorem}

We conclude the paper with some examples in which Theorem \ref{thm:main-thm} applies.

\medskip

\noindent\textbf{Note.} In the current version v3 we have corrected a small technical error of v2 where the singularity
of the orbifold symplectic form $\o_F$ at the origin was not handled correctly. To solve it,
we have interpolated $\o_F$ with $0$ near the singularity in a suitable way.
Some extra minor modifications have been made along the way, mainly notational issues.

\medskip

\noindent\textbf{Acknowledgements.} We are grateful to Fran Presas for useful conversations.
We also thank Luc\'ia Mart\'in-Merch\'an for pointing out a gap in the previous version v2 and for her
careful reading of the current version v3.
The authors were partially supported by Project MICINN (Spain) MTM2010-17389. 
The second author acknowledges financial support by the International PhD program La Caixa-Severo Ochoa.

\section{Orbifolds}\label{sec:orbif}

We start by giving the basic definitions and results of symplectic orbifolds that we will need later.

\begin{definition}\label{definition orbifold}
An $n$-dimensional (differentiable) orbifold is a Hausdorff and second-countable space $X$ 
endowed with an atlas $\{(U_{\a},V_\a, \phi_{\a},\G_{\a})\}$, 
where $\{V_\a\}$ is an open covering of $X$,
$U_\a \subset\RR^n$, $\G_{\a} < \Diff(U_\a)$ is a finite group acting by diffeomorphisms, and 
$\phi_{\a}:U_{\a} \to V_{\a} \subset X$ is a $\G_{\a}$-invariant map which induces a 
homeomorphism $U_{\a}/\G_\a \cong V_{\a}$. 

There is a condition of compatibility of charts for intersections.
For each point $x \in V_{\a} \cap V_{\b}$ there is some $V_\d \subset V_{\a} \cap V_{\b}$ 
with $x \in V_\d$ so that there are group monomorphisms $\rho_{\d \a}: \G_\d \inc \G_\a$,
$\rho_{\d \b}: \G_\d \inc \G_\b$, and open embeddings $\imath_{\d \a}: U_\d \to U_\a$, 
$\imath_{\d \b}: U_\d \to U_\b$, which satisfy 
$\imath_{\d \a}(\g(x))=\rho_{\d \a}(\g)(\imath_{\d \a}(x))$ and 
$\imath_{\d \b}(\g(x)) = \rho_{\d \b}(\g)(\imath_{\d \b}(x))$, for all $\g\in \G_\d$. 
\end{definition}

For an orbifold $X$, a change of charts is the map 
 $$
 \psi^{\d}_{\a \b}= \imath_{\d \b} \circ \imath_{\d \a}^{-1}:\imath_{\d \a}(U_\d) \subset U_\a
\to \imath_{\d \b}(U_\d) \subset U_\b.
 $$
So the change of charts between the chart $U_\a$ and $U_\b$ depends on the 
inclusion of a third chart $U_\d$. This dependence is up to the action of an element in $\G_\d$. In general,
we abuse notation and write $\psi_{\a \b}$ for any change of chart between $U_\a$ and $U_\b$.

For any point $x\in X$, by taking $U$ a small enough neighbourhood 
we can arrange always a chart  $U \subset \RR^n$, $U/\G \cong V$
so that the group $\G$ acting on $U$ leaves the point $x$ fixed, i.e.\ $\g(x)=x$ for all $\g \in \G$.
In this case, we call $\Gamma$ the \emph{isotropy group} at $x$, and we denote it by $\Gamma_x$.

We call $x\in X$ a \emph{smooth} point if a neighbourhood of $x$ is 
homeomorphic to a ball in $\RR^n$, and singular otherwise. 
We call $x \in X$ a \emph{regular} point if the isotropy group $\G_x=\{ \Id \}$ is trivial, and
we call it an \emph{isotropy} point if it is not regular.
Clearly a regular point is smooth, but not conversely.
We say that an orbifold $X$ is smooth if all its points are smooth. 
This is equivalent to $X$ being a topological manifold.
Finally, let us denote by $\S$ the set of isotropy points of an orbifold $X$.

\begin{proposition} \label{isometry local model}
Every orbifold $X$ has an atlas $\{(U_\a,V_\a, \phi_a,\G_\a)\}$ where the isotropy groups $\G_\a < \OO(n)$.
\end{proposition}

\begin{proof}
Let $\phi:U\to V \cong U/ \G$ be a small orbifold chart around a point $x \in X$, with $\G$ acting on $U\subset \RR^n$ 
by diffeomorphisms. We can suppose that the point $x=\phi(0)$ and that all elements of
$\G$ fix $0$. We consider the standard metric $g_{std}$ on $U$ and take $g:= \frac{1}{|\G|} \sum_{\g \in \G} \g^* g_{std}$.
Then $g$ is a Riemannian metric on $U$ and it is $\G$-invariant. We consider now 
the exponential map for the metric $g$, $\exp_0: T_0 U=\RR^n \to U$. Since any $\g \in \G$
acts by isometries, we have $\exp_0 \circ d_0 \g (v) = \g \circ \exp_0 (v)$ for all $v \in \RR^n$.
Take $\epsilon>0$ small enough so that $\exp_0:B_\epsilon(0) \to U'=\exp_0(B_\epsilon(0))\subset U$ is
a diffeomorphism. Then we have a chart $\phi'=\phi\circ \exp_0:B_\epsilon(0) \to V'=\phi(U')$ and
the group $\Gamma$ acts on $B_\epsilon(0)$ via $\Gamma \hookrightarrow \GL(n)$, $\gamma \mapsto d_0\gamma$.
Moreover, $d_0\gamma$ are isometries with respect to the metric $g$ at the point $0$, i.e.\ $g|_0$. If we
take an orthonormal basis of $\RR^n$ with respect to $g|_0$, then $\Gamma < \OO(n)$. 
\end{proof}

\begin{proposition} \label{prop:isotropy set}
Let $X$ be an orbifold, and let $\S$ be its isotropy subset. For every conjugacy class of
finite subgroup $H<\OO(n)$, we can define the set
 $$ 
 \S_H= \{ x\in X | \G_x \cong H\}.
 $$
Then the closure $\overline{\S}_H$ is an orbifold, and 
$\S_H= \overline{\S}_H -  \bigcup_{H<H'} \S_{H'}$ is a smooth manifold.
\end{proposition}

\begin{proof}
Let $x_0 \in \S$ be an isotropy point and take a local chart $(U,V,\phi,\G)$ near $x$ with $\G<\OO(n)$.
Let $\G=\{\g_1=\Id, \g_2,\ldots ,\g_N\}$ and consider the linear subspaces $L_i=\ker (\g_i-\Id)\subset \RR^n$, for 
$1\leq i\leq N$. For every subgroup $H<\G$, we define $L_H=\bigcap_{\gamma_i \in H} L_i\subset  \RR^n$.
This gives a finite collection of subspaces, which are stratified, in the sense that $H'<H$ implies that $L_H\subset L_{H'}$.
For given $H<\G$, let $L_H^0=L_H  -  \bigcup_{H'>H} L_{H'}$.
If $L_H^0$ is not empty,
then a point $x\in L_H^0$ satisfies that its isotropy is exactly $H$. So $\S_H\cap V =\phi(L_H^0\cap U)$.
Clearly $\overline{\S}_H=\phi(L_H\cap U)$, hence it is an orbifold with chart 
$(U\cap L_H,V\cap \overline{\S}_H, \phi, \G/\la H\ra)$.
Note that for any conjugate $\hat H=\gamma H \gamma^{-1}$, $L_{\hat H}=\gamma L_H$ and
$\phi(L_H\cap U)=\phi(L_{\hat H}\cap U)$, and the converse also holds. Take the minimal normal 
subgroup $\la H\ra$ containing $H$. Then $\G/ \la H \ra$ acts on $L_H$.
\end{proof}

An orbifold function $f:X\to \RR$ is a continuous function such that $f\circ \phi_\a:U_\a \to \RR$ is smooth
for every $\a$. Note that this is equivalent to giving smooth functions $f_\a$ on $U_\a$ which are $\G_\a$-equivariant
and which agree under the changes of charts.
An orbifold partition of unity subordinated to the open cover $\{V_\a\}$ of $X$ consists of orbifold functions
$\rho_\a:X \to [0,1]$ such that
the support of $\rho_\a$ lies inside $V_\a$ and the sum $\sum_\a \rho_\a \equiv 1$ on $X$.

\begin{proposition} \label{partition unity}
Let $X$ be an $n$-orbifold. For any sufficiently refined locally finite open cover $\{V_\a\}$ of $X$ there exists an
orbifold partition of unity subordinated to $\{V_\a\}$.
\end{proposition}

\begin{proof}
Take an open cover $\{V_\a\}$ of $X$ formed by coordinate patches $V_\a \cong B_{3\e}(0) /\G_\a$ with
$\G_\a < \OO(n)$ and so that $V'_\a \cong B_{\e}(0) /\G_\a$ is also an open cover of $X$. 
We can suppose that $V_\a$ is locally finite.
Take $\tilde{f}: \RR^n \to \RR$ be a radial bump function so that 
$\tilde{f} \equiv 0$ on $B_{3\e}(0)  -  B_{2\e}(0)$ and $\tilde{f} \equiv 1$ on $B_\e(0)$.
Since $\tilde{f}$ is a radial function and $\G_\a < \OO(n)$, it descends to the quotient 
and gives a continuous function $f_\a: V_\a \to \RR$ which can be extended by zero to all $X$
so we write $f_\a: X \to \RR$. The sum $\sum_{\b} f_\b(x)>0$ at all points of $X$ because the sets $V'_\a$ 
form a cover of $X$. We define  $\rho_\a= {f_\a}/{\sum_{\b} f_\b}$, and thus $\sum_\a \rho_\a \equiv 1$ on $X$. 
\end{proof}

Let $X$ be an orbifold with atlas $\{( U_\alpha,V_\a,\phi_\alpha,\Gamma_\alpha)\}$. 
An orbifold tensor on $X$ is a collection of tensors $T_\alpha$ 
on each $ U_\alpha$ which are $\G_\a$-invariant, and which agree under the changes of charts. 
In particular, we have the set of orbifold differential forms $\O_{orb}^p(X)$, orbifold Riemannian
metrics $g$, and orbifold almost complex structures $J$. The exterior differential, covariant derivatives, Lie bracket, 
Nijenhuis tensor, etc, are defined in the usual fashion.

\begin{proposition}
Let $X$ be an orbifold. There exists an orbifold Riemannian metric $g$ on $X$.
\end{proposition}

\begin{proof}
Let us consider an atlas $\{(U_\a,V_\a,\phi_\a, \G_\a)\}$ where the isotropy groups $\G_\a \subset \OO(n)$, 
whose existence is proved in Proposition \ref{isometry local model}.
Consider the standard metric $g_\a$ on $U_\a$ which is in particular $\G_\a$-invariant.
Take a differentiable partition of unity $\rho_\a$ subordinated to $\{V_\a\}$, given by Proposition \ref{partition unity}.
Define $g= \sum_\a \rho_\a g_\a$. This is an orbifold tensor on $X$, as $g_\a$ are orbifold tensors and $\rho_\a$ 
orbifold functions. It is an orbifold Riemannian metric by the usual convexity argument.
\end{proof}

An orbifold $X$ is orientable if all $\G_\a$ acts by orientation preserving diffeomorphisms and all embeddings $\imath_{\d\a}$
in Definition \ref{definition orbifold} preserve orientation. In this case we have an atlas with all $\G_\a<\SO(n)$ and
all changes of charts preserving orientation. This is equivalent to the existence of a globally non-zero orbifold form of
degree $n$, called a \emph{volume form}.

Given an orbifold $X$, the orbifold forms $(\Omega_{orb}(X),d)$ define the orbifold
De Rham cohomology algebra, and its cohomology is denoted $H_{orb}^*(X)$. 
This is isomorphic to the usual singular cohomology with real coefficients \cite{CFM}, 
 \begin{equation} \label{eqn:Horb}
 H^*_{orb}(X)\cong H^*(X,\RR).
 \end{equation}

\section{Symplectic orbifolds}\label{sec:2}

\begin{definition} \label{def:orb-sympl}
A symplectic orbifold $(X,\omega)$ is an orbifold $X$ equipped with an
orbifold $2$-form $\omega\in \Omega^2_{orb}(X)$ such that 
$d\omega=0$ and $\omega^n>0$, where $2n=\dim X$. In particular, it is oriented.

An almost K\"ahler orbifold $(X,J,\omega)$ consists of an orbifold $X$, and orbifold almost complex structure $J$ and
an orbifold symplectic form $\o$ such that $g(u,v)=\o(u,Jv)$ defines an orbifold Riemannian metric with $g(Ju,Jv)=g(u,v)$.

A K\"ahler orbifold is an almost K\"ahler orbifold satisfying the integrability condition that the Nijenhuis tensor $N_J=0$.
This is equivalent to requiring that the changes of charts are biholomorphisms of open sets of $\CC^n$.
\end{definition}

\begin{proposition} \label{almost Kahler-orbifold}
Let $(X,\o)$ be a symplectic orbifold. Then $(X,\o)$ admits an almost K\"ahler orbifold structure $(X, \o, J , {g})$.
\end{proposition}

\begin{proof}
Consider an auxiliary orbifold Riemannian metric $g_0$ on $X$. We define the orbifold endomorphism 
$A \in \End(TX)$ by the requirement $g_0(u,Av)= \omega(u,v)$.
The adjoint of $A$ with respect to $g$ is the orbifold endomorphism $A^* \in \End(TX)$ 
such that $g_0(u,A^*v)=g_0(Au,v)$. We have that $A^*=-A$ since 
$g_0(u,A^*v)=g_0(Au,v)=g_0(v,Au)= \o(v,u)=- \o(u,v)=-g_0(u,Av)=g_0(u,-Av)$.
The orbifold endomorphism $B=A A^*=-A^2$ is symmetric and positive. 
Indeed $g_0(u,Bu)=g_0(A^*u,A^*u) >0$ for $u\neq 0$, and $g_0(u,Bv)=g_0(A^*u,A^*v)=g_0(A^*v,A^*u)= g_0(v,Bu)$.

Let us see that $B$ admits a square root $\sqrt{B}\in \End(TX)$, which is an orbifold endomorphism.
On every chart $\phi:U\to V=U/\Gamma$, $B$ is given by a matrix valued function $B(x)$ on $U$ which is
$\Gamma$-equivariant. At every $x\in U$, it 
has positive eigenvalues and diagonalises, so we can define $\sqrt{B}$ locally 
as the matrix which has the same eigenvectors as $B$ with eigenvalues the (positive) 
square root of the eigenvalues of $B$. We have to see that $\sqrt{B}$ is $\G$-equivariant.
We take a real constant $\mu >0$ so that $\Vert \mu B - \Id \Vert <1$, in some operator norm,
so we have
 $$
 \sqrt{\mu} \sqrt{B}= \sqrt{\mu B}= \Id + \frac{1}{2} \mu B - \frac{1}{8} \mu^2 B^2 + \frac{1}{16} \mu^3 B^3 + \ldots
 $$
by the usual power series expansion of the square root. This 
yields the formula $\sqrt{B}= \frac{1}{\sqrt{\mu}} (\Id + \frac{1}{2} \mu B + \ldots)$.
As $\G$ commutes with $B$, we have that it also commutes with $\sqrt{B}$. 

Now define $J= -(\sqrt{B})^{-1} A$, which is an orbifold endomorphism.
As $\sqrt{B}=\sqrt{-A^2}$ commutes with $A$ by the power series expansion, 
its inverse $(\sqrt{B})^{-1}$ also commutes with $A$, and hence $J$ commutes with 
both $\sqrt{B}$ and $A$. Also $J^2= B^{-1} A^2= (-A^2)^{-1} A^2= - \Id$, so $J$ 
is an orbifold almost complex structure.
As $J^*= A^* \sqrt{B^*}= -A \sqrt{B}= - J$, we have that $g(u,v)=\omega(u,Jv)$ is a
symmetric bilinear orbifold tensor. 
Moreover 
 $$
 g(u,v)=\omega(u,Jv)=g_0(u,AJv)=g_0(u,(\sqrt{AA^*})^{-1} AA^*v)=g_0(u,\sqrt{AA^*}v),
 $$
which implies that $g$ is positive definite, and hence an orbifold Riemannian metric.
Finally, $J$ is compatible with $\o$ since 
$\o(Ju,Jv)=g(Ju,AJv)=g(J^*J u, A v)=g(u, Av)=\o(u,v)$.
So $(X,\o,{g}, J)$ is an almost K\"ahler orbifold.
\end{proof}

In the case of symplectic orbifolds, the structure of the isotropy set given in Proposition \ref{prop:isotropy set}
can be improved.

In the following, a \emph{$d$-suborbifold} of the orbifold $X$ is defined to be a connected subspace $Y \subset X$ such that
for each $y \in Y$ there is an orbifold chart $(U_\a,V_\a,\phi_\a,\G_\a)$ of $X$ around $y$ such that, calling $U'_\a = U_\a \cap (\RR^d \x \{0\})$, we have that $\phi_\a(U'_\a)=Y \cap V_\a$, $U'_\a$ is a $\G_\a$-invariant set
and moreover $U'_\a/\G_\a \cong Y \cap V_\a$. We say that $d$ is the dimension of $Y$.

\begin{remark}
\begin{enumerate}
\item This is one of many possible definitions of the concept of suborbifold. 
In other contexts, a less restrictive definition
is adopted, and the subspace $\S$ defined above is called a \emph{normalizable suborbifold}, or a \emph{full suborbifold}. See \cite{B-B,M}.

\item However, in this paper we are interested in the case that the ambient orbifold $X$ has \emph{homogeneous isotropy} (see definition \ref{def:HI}). In this case our notion of suborbifold coincides with the one given in \cite{B-B,M}.
\end{enumerate}
\end{remark}

\begin{corollary}\label{cor:overSH}
The isotropy set $\S$ of $(X,\o)$ consists of immersed symplectic suborbifolds $\overline{\S}_H$.
Moreover, if we endow $X$ with an almost K\"ahler orbifold structure $(\o,J,g)$, 
then the $\overline{\S}_H$ are almost K\"ahler suborbifolds.
\end{corollary}

\begin{proof}
Put any almost K\"ahler structure $(\o,J,g)$ on $X$ as provided by Proposition \ref{almost Kahler-orbifold}.
Fix a chart $(U,V,\phi,\G)$ with $\G < \OO(n)$, and $U \subset \RR^{2n}$ a neighborhood of $0$.
As $J$ is an orbifold almost complex structure, $\G$ preserves $J$, in particular $d_0\g \circ J_0=J_0\circ d_0\g$
for all $\g\in \G$. As $\g$ is linear, we have that $d_0\g=\g$, hence $\g$ preserves the complex structure 
of $\CC^n=(\RR^{2n},J_0)$. This means that $\G<\GL(n,\CC)\cap \OO(2n)=\UU(n)$.

As proved in Proposition \ref{prop:isotropy set}, the isotropy set $\S\cap V$ is the union of 
$\overline{\S}_H \cap V=\phi(U\cap L_H)$, for some subgroups $H<\G$. As $L_H=\bigcap_{\g\in H} L_{\g}$,
where $L_\g=\ker(\g-\Id)$, and $\g$ are complex endomorphisms, we have that $L_H$ is 
a complex linear subspace of $\CC^n$. This proves that $J_0$ leaves invariant $T_0 \overline{\S}_H=L_H$,
the (orbifold) tangent space of $\overline{\S}_H$ at the origin. This happens at every point, hence 
$\overline{\S}_H$ is an almost K\"ahler orbifold. In particular, it is a symplectic suborbifold of $(X,\o)$.
\end{proof}

The following result is a Darboux theorem for symplectic orbifolds.

\begin{proposition} \label{prop:Darboux}
Let $(X,\o)$ be a symplectic orbifold and $x_0 \in X$. There exists an orbifold chart $(U, V,\phi, \G)$ around $x_0$
with local coordinates $(x_1,y_1,\ldots, x_n,y_n)$ such that the symplectic form 
has the expression $\o= \sum dx_i \wedge dy_i$ and $\G < \UU(n)$ is a subgroup of the unitary group.
\end{proposition}

\begin{proof}
Take an initial orbifold chart $(U,V,\psi, \G)$ with $\G < \UU(n)$ and $x_0=\psi(0)$, possible
by Corollary \ref{cor:overSH}. Consider the evaluation of $\o$ at the origin $\o|_0$. 
We take a basis of $\RR^{2n}$ such that $\o|_0$ has standard form, that is
$\o|_0=\sum dx_i\wedge dy_i$. Let $\o_0$ be the symplectic form with constant coefficients
which equals to $\o|_0$.
Since $U$ is contractible we have that $\o - \o_0=d \mu$, for some $\mu \in \O^1(V)$.
We can suppose that $\mu$ is $\G$-invariant, since otherwise we put 
$\tilde{\mu}=\frac{1}{|\G|} \sum_{\g \in \G} \g^* \mu$ and $\tilde{\mu}$
also satisfies 
 $$
 d \tilde{\mu}= \frac{1}{|\G|} \sum_{\g \in \G} \g^* d \mu= \frac{1}{|\G|} \sum_{\g \in \G} \g^*( \o - \o_0) = \o - \o_0 \, .
 $$
We can further suppose that $\mu|_0=0$ vanishes as a $1$-form, since otherwise we put $\tilde{\mu}= \mu - \mu|_0$
which also satisfies $d \tilde{\mu} = \o - \o_0$ and $\tilde{\mu}$ is $\G$-equivariant.

Now we apply Moser trick. Consider $\o_t= t \o + (1-t) \o_0= \o_0 + t \, d \mu$.
Consider a vector field $X_t$ such that $\iota_{X_t} \o_t= - \mu$.
Let us call $\varphi_t$ the flow of the vector field $X_t$ at time $t$,
which satisfies $\frac{d}{dt}\varphi_t(x)=X_t|_{\varphi_t(x)}$ for each $x \in U$.
Then for each $s$,
 \begin{align*}
 \frac{d}{dt}\Big|_{t=s} \varphi_t^* \o_t &= \frac{d}{dt}\Big|_{t=s} \varphi_t^* \o_s 
 + \varphi_s^*\left( \frac{d}{dt}\Big|_{t=s} \o_t\right)
 = \varphi_s^*(\mathcal{L}_{X_s} \o_s) +  \varphi_s^* (d \mu) \\
 &= \varphi_s^* \left(d(\iota_{X_s} \o_s) + \iota_{X_s} d \o_s\right) + \varphi_s^* (d \mu) 
 = -\varphi_s^*( d\mu) + \varphi_s^* (d \mu) =0,
 \end{align*}
using Cartan formula for the Lie derivative $\mathcal{L}_X= d \iota_X + \iota_X d$.
This implies that $\o_0= \varphi_0^* \o_0 = \varphi_1^* \o_1= \varphi_1^* \o$.
The change of coordinates is then given by the diffeomorphism $\varphi:= \varphi_1$ 
which is defined in some neighborhood of $0 \in U$.
Recall that, since $\mu$ vanishes at $0 \in U$, $\varphi_t(0)=0$ for all $t$, so $\varphi(0)=0$.
Finally, as $\mu$ and $\o_t$ are $\G$-equivariant, and $\iota_{X_t}\o_t=-\mu$, we have
that the vector fields $X_t$ are $\G$-equivariant. Therefore the flow $\varphi_t$ are
$\G$-equivariant diffeomorphisms, and so $\varphi$ is $\G$-equivariant.
Summarising, we have a diffeomorphism $\varphi: U' \to U$ between two neighborhoods of $0$ 
and $\varphi^*\o = \o_0$ is a constant
symplectic form on $U'$. Moreover, since $\varphi \g \varphi^{-1}= \g$ for all $\g \in \G$, the $\G$-action induced
by $\varphi$ on $U'$ is the same as on $U$. The sought orbifold chart is $(U',V,\varphi\circ \psi,\G)$.
\end{proof}

\begin{corollary} \label{Complex local model}
Let $(X,\o)$ be a symplectic orbifold. Then $(X,\o)$ admits a \emph{Darboux orbifold atlas}, i.e.\
an atlas $\{(U_\a,V_\a,\phi_\a, \G_\a)\}$ where all the isotropy groups $\G_\a < \UU(n)$ 
and the expression in coordinates of $\o$ on each $U_\a \subset \RR^{2n}$
is the canonical form of $\RR^{2n}$, i.e.\
$\o|_{U_\a}= \sum d x_j \wedge d y_j= \frac{i}{2} \sum dz_j \wedge d\bar{z}_j$.

Moreover, if $\overline{\S}_H \subset X$ is an isotropy suborbifold of codimension $2k$, 
we can arrange that for each open set $V_\a$ which intersects $\overline{\S}_H$,
the intersection $\overline{\S}_H \cap V_\a$ is given by $\{z_1=0, \ldots, z_k=0\} \subset U_\a$.
\end{corollary}

\begin{proof}
By Proposition \ref{prop:Darboux}, there is a Darboux atlas as required.
Let us see that it can be adapted to the submanifold $D$. For each chart $(U_\a,V_\a,\phi_\a,\G_\a)$ 
intersecting $D$, $D \cap V_\a=\phi_\a(L_H\cap U_\a)$, where $L_H \subset \CC^n$ is a complex linear subspace, 
being the fixed subset of $\G_\a$. We can then take a unitary basis of $\CC^n$ so that 
$L_H=\{z_1=0, \ldots, z_k=0\}$, and clearly the symplectic form is again $\omega_0$ since $\UU(n)<\Sp(2n,\RR)$.
\end{proof}

\section{Tubular neighbourhood of the isotropy set}

From now on we restrict to the case where the isotropy locus $\S$ is already a smooth submanifold.

\begin{definition} \label{def:HI}
We say that an isotropy subset $\overline{\S}_H$ is \emph{homogeneous} if $\overline{\S}_H=\S_H$.
That is, all its points have isotropy equal to $H$.

An orbifold $X$ is called HI (abbreviature for homogeneous isotropy) if all
its isotropy subsets are homogeneous.
\end{definition} 

Note that by Proposition \ref{prop:isotropy set}, if $\overline{\S}_H$ is homogeneous, then it is a submanifold in the sense that the intrinsic isotropy groups of the suborbifold are the identity,
hence it has an orbifold structure without isotropy, i.e. a manifold structure.
From now on we work, unless otherwise stated, with a HI orbifold $X$.

\begin{lemma}
If $\overline{\S}_H$ is an homogeneous isotropy set, then it is isolated, that is, no other
isotropy set intersects it. Moreover, around any point $x_0\in \overline{\S}_H$ we
have a chart $(U,V,\phi,H)$, where $U \cong U'\x U''$, $U'\subset \RR^d$,
$U''\subset \RR^{n-d}$, $H<\OO(n-d)$, where $d$ is
the dimension of $\overline{\S}_H$,  $V\cong U' \x (U''/H)$, 
and $\overline{\S}_H$ corresponds to $U'\x \{0\}$.

If $(X,\o)$ is a symplectic orbifold of dimension $2n$ and $\overline{\S}_H$ is an homogeneous
isotropy set of dimension $2d$, then for every $x \in \overline{\S}_H$ there is a 
Darboux chart $(U,V,\phi,H)$ around $x$, where 
$U \cong U'\x U''$, $U'\subset \CC^d$,
$U''\subset \CC^{n-d}$, $H<\UU(n-d)$,  $V\cong U' \x (U''/H)$, 
and $\overline{\S}_H$ corresponds to $U'\x \{0\}$.
\end{lemma}

\begin{proof}
We have $\overline{\S}_H\cap V=\phi(L_H\cap U)$. The linear subspace $L_H$ is $d$-dimensional,
so we can write $\RR^n=L_H\oplus (L_H)^\perp$. Note that $\G$ fixes $L_H$, so it acts
on $(L_H)^\perp \cong \RR^{n-k}$. Moreover $\G=H$. The result follows.

The statement for symplectic orbifolds follows analogously using Corollary \ref{Complex local model}.
\end{proof}

To understand the structure around an homogeneous isotropy subset, let us introduce the 
notion of orbifold bundle, as bundle of orbifolds over a manifold.
For a space with a geometric structure $(M,G)$ we understand a smooth manifold $M$ with a Lie
group $G$ acting on $M$. We call $G$ the automorphism group of the structure and write
$G=\Aut(M)$.

\begin{definition} \label{def:orbifold-bundle}
Let $M$ be a space with some geometric structure and let $\G < \Aut(M)$
be a finite subgroup of automorphisms of $M$, and let $B$ be a smooth manifold.
An orbifold bundle $E$ with fiber $F=M/\G$ and base space $B$ consists of an orbifold $E$
endowed with an open cover $\{ V_\a \}$ and with orbifold charts 
$\phi_\a: U_\a \times M \to V_\a$ so that:
\begin{enumerate}
\item The groups $\G_\a<\Aut(M)$ act on $U_\a \x M$ as $\g (x,m)=(x,\g m)$ for all $\g \in \G_\a$.

\item All the groups $\G_\a$ are conjugated to $\G$ by some automorphism of $M$, so all the quotients
$M/\G_\a$ are isomorphic to $F=M/\G$.

\item The changes of charts of this atlas of $E$ are maps of the form
 $$
 \varphi_{\a \b}: \imath_{\d \a} (U_\d) \times M \to \imath_{\d \b} (U_\d) \times M, (x,m) \to (\psi_{\a \b}(x), A_{\a \b}(x) m),
 $$
with $A_{\a \b}: \imath_{\d \a} (U_\d) \to \Aut(M)$ is a smooth map taking values in the group of automorphisms of $M$.
\end{enumerate}
\end{definition}

Note that from the definition of orbifold, the maps $A_{\a \b}$ are compatible with the actions of the local groups $\G_\a$
and $\G_\b$ in the sense that $A_{\a \b}(x) \g m= \rho_{\a \b}(\g) A_{\a \b}(x) m$ for all $\g \in \G_\a$, where
$\rho_{\a\b}=\rho_{\d\b}\circ \rho_{\d\a}^{-1}:\G_\a \to \G_\b$ are all group isomorphisms. Note
that it must be $\rho_{\a \b}(\g)=A_{\a \b}(x) \g A_{\a \b}(x)^{-1}$, so (2) in Definition \ref{def:orbifold-bundle} is
automatic.

An orbifold bundle satisfies that $E$ is topologically a fiber bundle of the form $F=M/\G \to E \to B$. 
The transition functions are induced by $A_{\a\b}$ on $M/\G$.

A vector orbifold bundle corresponds to the case where $M$ is a (real or complex)
vector space and $\Aut(M)$ is a subgroup of the group of linear maps of $M$.

Now let $(X,\o)$ be a HI symplectic orbifold, and let 
$D=\overline{\S}_H$ be an homogeneous isotropy set of dimension $2d$.
Let $2k=2n-2d$ be the codimension of $D$.
The orbifold tangent space $TX$ is given in local charts $(U,V,\phi,\G)$ by 
$T_xU$ with the action of $\G_x<\GL(T_xU)$ induced by $d_x\g$, for
$\g\in \G$ acting on $U$. If $x\in D$, then $T_xU$ is a symplectic
vector space and $T_xD$ is the fix set of $\G_x$. The symplectic
orthogonal $(T_xD)^{\perp \omega}\cong \RR^{2k}$ has the action
induced by $\G_x$, and we define the orbifold normal space as
$$
 \nu_{D,x}=(T_xD)^{\perp \omega}/\G_x\, .
 $$
The normal bundle $\nu_D$ is the union of all $\nu_{D,x}$, for $x\in D$.

\begin{proposition}\label{prop:normal}
Let $(X,\o)$ be a HI symplectic  orbifold, and let $D\subset X$ be an isotropy submanifold. Then the
normal bundle $\nu_D$ admits the structure of a symplectic orbifold vector bundle over $D$.
\end{proposition}

\begin{proof}
We take a collection of symplectic charts $(U_\a,V_\a,\phi_\a,\G_\a)$ adapted to $D$,
given by Corollary \ref{Complex local model}. Denote $2d=\dim D$
and let $2k=2n-2d$ be the codimension of $D$.  Then 
$U_\a = U_\a' \x U_\a''$, where $U_\a' \subset \CC^{d}$,
$U_\a'' \subset \CC^{k}$, $\G_\a < \UU(k)$, and $\phi_\a: V_\a\cong U_\a' \x (U_\a''/\G_\a)=V'_\a \x V''_\a$. 
Then $\phi_\a:U_\a'\to V_\a'$ is a diffeomorphim, and $\{V_\a'\}$ is a covering by charts of $D$.

For any $p \in U_\a' \subset D$, the tangent space $T_pD=\CC^d\x\{0\}$ and
$(T_pD)^{\perp\omega}=\{0\}\x \CC^k$. Therefore $\nu_D|_{U_\a'}\cong U_{\a}' \x (\CC^k/\G_\a)$,
where  $\nu_D|_{U_\a'}$ denotes the collection of normal spaces to points $p \in U_{\a}'$. 
Then there is an orbifold chart
  $$
  U_\a' \x \CC^k  \to \nu_D|_{U_\a'}\, ,
 $$
where $\G_\a$ acts on $\CC^k$ by the inclusion $\G_\a <\UU(k)$. The fiber is  $M=\CC^k$ 
with $\Aut(M)=\UU(k)$.
Let us see that the orbifold changes of charts satisfy (3) in Definiton \ref{def:orbifold-bundle}. 
By Definition \ref{definition orbifold}, the change of charts for $U_\a$ and $U_\b$ is given by a
map 
 $$
 \psi_{\a\b}: \imath_{\d\a}(U_\d'\x U_\d'') \to \imath_{\d\b}(U_\d'\x U_\d''), \quad 
 \psi_{\a\b}(x,y)=( \psi_{\a\b}'(x,y), \psi_{\a\b}''(x,y)).
 $$
The group homomorphisms
$\rho_{\d \a}: \G_\d \inc \G_\a$ and 
$\rho_{\d \b}: \G_\d \inc \G_\b$ are isomorphisms (since all points have the same isotropy), so the
map $\rho_{\a\b}= \rho_{\d\b}\circ\rho_{\d\a}^{-1}: \G_\a \to \G_\b$ is an isomorphism.
The map $\psi_{\a\b}$ satisties $\psi_{\a\b}(x,\g y)=\rho_{\a\b}(\g)(\psi_{\a\b}(x,y))$, i.e.
 \begin{equation}\label{eqn:1}
 \psi_{\a\b}''(x,\g y)=\rho_{\a\b}(\g) \psi_{\a\b}''(x,y),
 \end{equation}
for $\g\in \G_\a$. 
Take a point $x=(x,0)\in U_\a'\subset U_\a$. The map at the tangent space $T_xX$ is
given by $(d\psi_{\a\b})_{(x,0)}$. Therefore the induced map on $(T_xD)^{\perp \omega}=\{0\}\x \CC^k$
is given by the differential in the direction of $y$, which is
 $$
 A_{\a\b}(x)=\frac{\bd \psi_{\a\b}''}{\bd y}\Big|_{(x,0)}\, .
 $$
By differentiating (\ref{eqn:1}), we have $A_{\a\b}(x)\g m=\rho_{\a\b}(\g) A_{\a\b}(x)m$, for
$m\in \CC^k$. Note that $A_{\a\b}(x)\in \Sp(2k,\RR)$, since $\psi_{\a\b}$ are symplectomorphisms.
We consider the geometric space
$M=\CC^k$ with group $\Aut(M)=\Sp(2k,\RR)$. This completes the proof.
\end{proof}

\begin{proposition}[Tubular neighbourhood for orbifolds] \label{prop:tubular}
Let $X$ be an orbifold and $D \subset X$ an homogeneous isotropy submanifold.
Then there exists a tubular neighborhood of $D$ in $X$ which is diffeomorphic (as orbifolds) 
to a neighborhood of the zero section of the orbifold normal bundle $\n_D$.
\end{proposition}

\begin{proof}
Consider an orbifold Riemannian metric $g$ for $X$.
We use the exponential map associated to the metric to find the desired diffeomorphism.
Take the normal bundle $\nu_D=\{(x,u) | u \in (T_{(x,0)}D)^{\perp} \}$ and let 
$D=D \times \{0\}\subset\nu_D$ be the zero section.
Define $\exp: \nu_D \to U/\G \subset X$ by $\exp(x,[u]) = [\a_{((x,0),u)}(1)]$, where $\a_{((x,0),u)}$ is
the geodesic from $(x,0) \in U$ with direction $u$. 
The brackets stand for the equivalence classes modulo the local isotropy groups.
We have to see that the map $\exp$ is defined locally in each orbifold chart,
$\exp:\nu_D|_{U'}\to U/\G=U'\x (U''/\G)$,
and it is $\G$-equivariant. The isotropy groups $\G$ act by isometries on the orbifold charts
and hence commute with the exponential map, so $\exp(x,\g u)=\g (\exp (x,u))$ for $\g\in \G$. 
There are open sets $\cU,\cV$ with $D\subset \cU\subset \n_D$, $D\subset \cV \subset M$, so
that $\exp:\cU\to \cV$ is defined. As $\exp$ is the identity on $D$, it yields an orbifold diffeomorphism 
$\exp:\cU\to \cV$ for small open sets.
\end{proof}

The next result is the orbifold version of the tubular neighbourhood theorem for
symplectic submanifolds.

\begin{proposition}[Symplectic tubular neighborhood for orbifolds] \label{prop:tubular-sympl}
Let $(X,\o)$ be a symplectic orbifold and let $D \subset X$ be an homogeneous isotropy submanifold. 
Let $\cU \subset \nu_D$ be a neighborhood of $D$ in the orbifold normal bundle $\nu_D$.
Suppose that $(\cU, \o_{\nu_D})$ is a symplectic structure on $\cU$ such that the symplectic form $\o_{\nu_D}$ satisfies
that $\o_{\nu_D}$ and $\o$ coincide on $T_x X$ for all points $x \in D$, 
(here $D$ is identified to the zero section of $\nu_D$).
Then there are open sets $\cU',\cV'$ with $D\subset \cU'\subset \cU\subset \nu_D$ and $D\subset \cV'\subset X$
and an orbifold symplectomorphism $\varphi: (\cU', \o_{\nu_D}) \to (\cV', \o)$ so that $\varphi|_D = \Id_D$.
\end{proposition}

\begin{proof}
The proof is similar to the equivariant Darboux theorem (Proposition \ref{prop:Darboux}).
Take first any orbifold diffeomorphism $h:\cU \subset \nu_D \to \cV \subset X$ such that $h|_D=\Id_D$
by Proposition \ref{prop:tubular} (maybe reducing $\cU$ if necessary).
Let us call $i: D \to \nu_D$ the inclusion of $D$ as the zero section, and let
$\o_0= \tilde{\o}$, $\o_1= h^*(\o)$, so that $\o_0$ and $\o_1$ are two symplectic 
forms on $\cU \subset \nu_D$ such that 
$i^*(\o_1 - \o_0)=0$. 

By (\ref{eqn:Horb}), the orbifold De Rham cohomology $H_{orb}^2(\nu_D) \cong H^2(\nu_D)$.
Hence the inclusion $i: D \to \nu_D$ 
induces an isomorphism $i^*: H_{orb}^2(\nu_D) \to H^2(D)$. 
So there exists an orbifold one form $\mu \in \O_{orb}^1(V)$ such that $d \mu = \o_1 - \o_0$. 
We can suppose that the restriction $i^* \mu$ of $\mu$ to the zero section vanishes.
Indeed, if not then we would consider the form $\tilde{\mu}=\mu - \pi^* i^* \mu$
which also satisfies $d \tilde{\mu}= d \mu - \pi^* i^*(\o_1 - \o_0)= d \mu= \o_1 - \o_0$,
and $i^* \tilde{\mu}=i^* \mu - i^* \pi^* i^* \mu = i^* \mu - i^* \mu =0$.

Consider the form $\o_t= t \o_1 + (1-t) \o_0= \o_0 + t \, d \mu$, for $0 \le t \le 1$. 
Since $i^*\o_t= i^*\o_0= i^* \o_1$ is symplectic on the zero section $D$, we can suppose, reducing $\cU$ if necessary,
that $\o_t$ is symplectic on some neighborhood, which we call $\cU$ again, of the zero section $D$ of $\nu_D$.
The equation $\iota_{X_t} \o_t= - \mu$ admits a unique solution $X_t$ which is a vector field on $V$.
Since $i^* \mu=0$, it follows that $X_t|_x=0$ for every $x \in D \subset \nu_D$.
Now consider the flow $\varphi_t$ of the family of vector fields $X_t$.
There is some $\cU'\subset \cU$ such that $\varphi_t: \cU' \to \cU$ for all $t \in [0,1]$.
Moreover $\varphi_0= \Id_{\cU}$, and $\varphi|_D= \Id_D$. 
We compute
 \begin{align*}
 \frac{d}{dt}\Big|_{t=s} \varphi_t^* \o_t &= 
  \varphi_s^*\left(\mathcal{L}_{X_s} \o_s\right) +  \varphi_s^* (d \mu)  \\
 &= \varphi_s^* \left(d(\iota_{X_s} \o_s) - \iota_{X_s} d \o_s \right) + \varphi_s^* d \mu 
 = -\varphi_s^* (d\mu) + \varphi_s^* (d \mu) =0.
 \end{align*}
This implies that $\o_0= \varphi_0^* \o_0 = \varphi_1^* \o_1$.
So $\varphi_1:(\cU',\o_{\nu_D}) \to (\cU,h^*(\o))$ is a symplectomorphism.
It remains to see that $\varphi$ is $\G_\a$-equivariant by all the local isotropy groups $\G_\a$.
Fix a chart of $\nu_D$ and suppose that the group $\G$ acts on this chart.
As $\omega_t$ and $\mu$ are $\G$-equivariant, we have that $X_t$ are $\G$-equivariant. This implies that
the diffeomorphisms $\varphi_t$ are $\G$-equivariant.

Given $\varphi=\varphi_1$ as above, take the composition 
$\psi=h\circ \varphi: (\cU',\o_{\nu_D}) \to (\cV,\o)$, which is our desired orbifold symplectomorphism
of $\cU'$ onto $\cV'=\psi(\cU')\subset \cV$.
\end{proof}

Now let $(X,\o)$ be a symplectic orbifold with an homogeneous isotropy submanifold $D\subset X$.
Let $2d$ be the dimension of $D$ and $2k=2n-2d$ its codimension. 
Then we take $(\o,g,J)$ any orbifold almost K\"ahler structure for $(X,\o)$.  
For $x_0 \in D$, we take an orbifold Darboux chart $(U,V,\phi, \G)$ adapted to $D$, with $\G<\UU(k)$.
So the lifting of $D$ to $U$ is given by $\{ z_{d+1}=0, \ldots ,z_n=0 \}$.
By compatibility of $g$ and $\omega$,
we have $(T_{x_0}D)^{\perp \o}= (T_{x_0}D)^{\perp g}$, and it has the structure of a $J$-complex
subspace of $T_{x_0} U = \CC^n$, and it is given by $(T_{x_0}D)^{\perp \o}=\{ z_1=0, \ldots, z_d=0 \}$.
The action of $\G$ on $U$ is given by $\g (x,y)=(x,\g y)$ for $x=(z_1,\dots,z_{d}) \in \CC^{d}$ 
and $y=(z_{d+1},\dots,z_n) \in \CC^{k}$.

Under the diffeomorphism $F:\cU\to \cV$ provided by Proposition \ref{prop:tubular}, where
$\cU$ is a neighbourhood of the zero section $D \subset \n_D$ and $\cV$ is a neighbourhood
of $D\subset X$, we can consider the pull-back of $\o$ to $\cU$, which we will call
$\o$ again. So $\o\in \O^2_{orb}(\cU)$ is a symplectic orbifold form.

The following proposition modifies a symplectic form on a normal 
bundle of a symplectic suborbifold so as to make it linear in the fibers.

\begin{proposition} \label{prop:linear-symplectic}
Let $(X,\o)$ be a symplectic orbifold and $D$ a homogeneous isotropy submanifold. Denote $(D,\o_D)$
the inherited symplectic structure. 
The total space of the bundle $\pi: \nu_D \to D$ admits 
an orbifold closed $2$-form $\o_{\nu_D}$ such that:
 \begin{itemize}
 \item $\o_{\nu_D}$ and $\o$ coincide along the zero section $D\subset \nu_D$,
 in particular $\o_{\nu_D}$ is symplectic on an open set $\cU$ with $D\subset  \cU\subset \nu_D$.
 
 \item Restricted to any fiber $F_x=\nu_{D,x}=(T_xD)^{\perp \o}/\G_x$, the form
 $\o_{\nu_D}|_{F_x}$ is constant on the vector space $(T_xD)^{\perp\o}$.
 \end{itemize}
\end{proposition}

\begin{proof}
We consider a local trivialization of $\nu_D$, given by a chart $\phi:U_\a \x \CC^k  \to \nu_D|_{U_\a}$
with coordinates $(x,y)$, $x \in U_\a, y \in \CC^k$, and
with group $\G_\a < \UU(k)$. Consider the form $(\o_x)|_{(T_xD)^{\perp\o}}=(\o_x)_{F_x}$,
which is a $\G_\a$-equivariant symplectic $2$-form on the vector space $(T_xD)^{\perp\o}$.
Write $(\o_x)_{F_x}=\sum b_{ij}(x) dy_i\wedge dy_j$ and let $\b_\a= d(\sum b_{ij}(x) y_i dy_j)=d \eta_\a$.
Then $\b_\a$ is closed and satisfies $\b_\a|_{F_x}=(\o_x)_{F_x}$ for every $x\in D$.
Averaging over $\G_\a$, we have a $\G_\a$-invariant form $\tilde \b_\a$ satisfying the same conditions.
Now consider $\o'_\a=\pi^*(\o_D) +\tilde \b_\a$. 
This is $\G_\a$-invariant, $(\o'_\a)_x=\o_x$ for all $x\in U_\a$ and it is constant on fibers.
Clearly $\o'_\a=\pi^*(\o_D) +d \tilde \eta_\a$, for  
$\tilde \eta_\a= \in \O^1(U_\a \times \CC^k)$ the average of $\eta_\a=\sum b_{ij}(x) y_i dy_j$. 
Note that the $2$-forms $d \eta_\a$ restrict to $0$ on $U_\a \times \{0\}$
and restrict to $(\o_x)_{F_x}$ on every fiber $F_x$ over a point $x \in U_\a$.

Take any smooth orbifold partition of unity $\rho_\a$ subordinated to the cover $U_\a$ of $D$. 
Consider the form 
 \begin{equation}\label{eqn:refgive}
 \o_{\nu_D}= \pi^*(\o_D) +  \sum_{\a} d \left( (\pi^*\rho_\a) \tilde \eta_\a \right) = \pi^*(\o_D) + \tilde \b .
 \end{equation}
Note that $\o_{\nu_D}$ is invariant by the local groups since all objects involved in its definition are. 
Restricting to a fiber $F_x$, we have $\o_{\nu_D}|_{F_x}=\sum d (\rho_\a(x) \tilde \eta_\a)=
\sum \rho_\a(x) (\o_x)_{F_x}=(\o_x)_{F_x}$. For $(x,0) \in \nu_D$, we have from the expression
$\o_{\nu_D}= \pi^*(\o_D) + \sum d (\pi^*\rho_\a) \wedge \tilde \eta_\a + \sum (\pi^*\rho_\a) d \tilde \eta_\a$
and the fact that $\tilde \eta_\a$ vanishes at $(x,0)$, that $\o_{\nu_D}|_{(x,0)}=\o|_{(x,0)}$.
In particular, $\o_{\nu_D}$ is non-degenerate at every point $(x,0)$ in the zero section, which implies that
$\o_{\nu_D}$ is also non-degenerate in some open neighborhood $\cU$ of the zero section in $\nu_D$.
Since $\o_{\nu_D}$ is closed, it is symplectic on $\cU$. 
\end{proof}
The following is the analogous result of Proposition \ref{prop:linear-symplectic} in the almost K\"ahler setting.
\begin{proposition}\label{prop:linear-Kahler}
Let $(X,\o, J,g)$ be an almost K\"ahler orbifold and $D$ a homogeneous isotropy submanifold
with its inherited structure $(D,\o_D,J_D,g_D)$.
An open neighborhood $V \subset \nu_D$ of the zero section $D=D \times \{0\}\subset \nu_D$ 
admits an orbifold almost K\"ahler 
structure $(\o_{\nu_D}, J_{\nu_D}, g_{\nu_D})$  such that:
\begin{itemize}

\item For a point $(x,0)$ in the zero-section we have that, under
the natural splitting $T_{(x,0)}(\nu_D)= T_x D \oplus (T_x D)^{\perp}$, 
the restriction of $(\o_{\nu_D}, J_{\nu_D})$ to $T_x D$ and $(T_x D)^{\perp}$ coincides with $(\o, J)$.

\item The tensors $\o_{\nu_D}$, $J_{\nu_D}$ and $g_{\nu_D}$ are constant along the fibers $F_x=\nu_{D,x}$, for $x\in D$. 
\end{itemize}
\end{proposition}
\begin{proof}
We take the symplectic structure $\o_{\nu_D}$ provided by Proposition \ref{prop:linear-symplectic}. 
Let us define first an auxiliar metric $g'$ on $V \subset \nu_D$.
We define $g'$ so that is coincides with $g$ on $T_xD$ and on $(T_xD)^\perp$ for $x\in D$.
On the fiber $F_x=\nu_{D,x}=((T_xD)^\perp)/\G_x$, the tensors $g_x|_{(T_xD)^\perp}$ and 
$J_x|_{(T_xD)^\perp}$ are $\G_x$-equivariant, so we can define constant tensors on $F_x$, which
vary smoothly for $x\in D$. Define $g'_y$ equal to $g_x|_{(T_xD)^\perp}$ at any point $y \in F_x$.

Now we extend $g'$ to a Riemannian metric on $V\subset \n_D$. This is done as follows.
For $(x,u)\in V\subset \n_D$, with $u\neq 0$, we consider the splitting $T_{(x,u)} \nu_D=T_{u} F_x \oplus
(T_{u} F_x)^{\perp \o_{\nu_D}}$. 
We define $g'$ by making these subspaces orthogonal
so that $g'$ restricted to $(T_{u} F_x)^{\perp \o_{\nu_D}}$
is $\pi^*(g|_{T_xD})$ under the isomorphism $\pi_*:(T_{u} F_x)^{\perp \o_{\nu_D}} \to T_xD$. 
The metric $g'$ may not be equivariant, so we make it equivariant by averaging
and then we use the method of the proof of Proposition \ref{almost Kahler-orbifold}
to modify $g'$ into an orbifold Riemannian metric $g_{\nu_D}$ such that $g_{\nu_D}(u,v)=\o_{\nu_D}(u,J_{\nu_D}v)$ defines 
an orbifold almost K\"ahler structure $J_{\nu_D}$. Note that the tensor $A$ defined by $g'(u,Av)=\o_{\nu_D}(u,v)$
satisfies that $A= J$ at the points of $D \subset \nu_D$, so $J_{\nu_D}=J$ on $D$ as desired. 

Let us see now that $J_{\nu_D}=J$ in $TF$. Take any $y \in \nu_D$ and take tangent vectors 
$u \in (T_y F)^{\perp \o{\nu_D}}=(T_yF)^{\perp g'}$, $v \in T_yF$,
so $g'|_y(u,Av)=\o_{\nu_D}|_y(u,v)=0$, and it follows that $Av \in  T_yF$. 
Now, for $u, v \in T_yF$ we have $g'|_y(u,Av)=g_{\pi(y)}(u,Av)=g_{\pi(y)}(u,Jv)=g'_y(u,Jv)$ since we have defined $g'$
so that it coincides with $g$ in $TF$ and we already saw that $A=J$ on points of $D$. 
It follows that $A|_y v=J_{\pi(y)}v$ for all $y \in \nu_D$ and $v \in T_y \nu_D$, which implies that also
$J_{\nu_D}|_y v=J_{\pi(y)}v$, so $J_{\nu_D}$ is constant along $F_x$ as desired.
Recall that $g_{\nu_D}$ is determined by $J_{\nu_D}$ and $\o_{\nu_D}$ so it must be also constant on the fibers.
This concludes the proof.
\end{proof}

To proceed further, we will use the natural retraction of \cite[Prop.\ 2.2.4]{MS},
 \begin{equation} \label{eqn:r(A)}
  r:\Sp(2k,\RR)\to \UU(k), \qquad r(A)=A(A^tA)^{-1/2}
 \end{equation}
We note that there is a group $\G<\UU(k)$ and an isomorphism $\rho:\G \to \G'<\UU(k)$, such that
$A$ is $\G$-equivariant, in the sense that if $A\circ \g=\rho(\g)\circ A$, 
then $r(A)$ is also $\G$-equivariant.
\begin{lemma} \label{lem:retrac}
Let $A,C \in \UU(k)$ and $B \in \Sp(2k,\RR)$ such that $A=B^{-1}  C B$.
Then $A=r(B)^{-1} C \, r(B)$.
\end{lemma}

\begin{proof}
The fact $B \in \Sp(2k,\RR)$ means that $B^t J_0 B=J_0$, where $J_0$ is the matrix of the standard complex structure.
So $B^t=-J_0 B^{-1} J_0$, $A^tA=C^tC=\Id$, $A J_0= J_0 A$ and $C J_0= J_0 C$. Then
 \begin{align*}
 (B^t B) A &=- J_0 B^{-1} J_0 B A=- J_0 B^{-1} J_0 C B=- J_0 B^{-1} C J_0 B  \\
  &=- J_0 A B^{-1} J_0 B=- A J_0 B^{-1} J_0 B = A (B^t B).
 \end{align*}
This means that $A$ commutes with $B^t B$. Therefore  $A$ commutes with $(B^t B)^{1/2}$ as well. 
Hence $r(B)^{-1} C r(B)=(B^t B)^{1/2} B^{-1} C B (B^t B)^{-1/2} = 
(B^t B)^{1/2} A (B^t B)^{-1/2}=A$, as required.
\end{proof}

\begin{proposition}
The normal orbifold bundle $\nu_D$ admits an atlas such that the transition functions $A_{\a \b}$
are $\UU(k)$-valued. 
In the terminology of Definition \ref{def:orbifold-bundle}, the structure group of $\nu_D$ reduces to $\UU(k)$.
\end{proposition}

\begin{proof}
By Propositions \ref{prop:normal} and \ref{prop:linear-Kahler}, the normal orbifold bundle $\nu_D$ 
admits an almost K\"ahler structure $(\o, J, g)$ which is constant along the fibers, 
and it also admits the structure of a $\Sp(2k,\RR)$-orbifold bundle. Call $h$ the hermitian metric
associated with $(\o,J,g)$.
Take an atlas $\{(U_\a \x \CC^k, \G_\a\ , \o_0)\}_{{\a \in I}}$ of $\nu_D$ so that $\G_\a < \UU(k)$, $\o_0$
the standard symplectic form in $\CC^k$, and the transition functions 
are $A_{\a\b}: U_\a \cap U_\b \to \Sp(2k,\RR)$.
 
Fix a chart $U_\a \x \CC^k$ and call $(x,y)$ the corresponding coordinates. 
The hermitian metric $h$ induces a linear hermitian metric $h_x$ on 
each fiber $\{x\} \x \CC^k$ varying smoothly with $x \in U_\a$. Using a $h_x$-unitary frame, 
this is determined by a matrix $C_\a(x) \in \Sp(2k,\RR)$. 
The orbifold almost K\"ahler structure of the fibers is given by tensors $(\o_0,J_x, g_x)$, which
are $\G_\a$-equivariant. If we introduce new coordinates 
$(x,\tilde{y})=(x, C_\a(x) y)$ then the orbifold almost K\"ahler structure of the fibers is given
by the standard tensors $(\o_0,J_0, g_0)$ defining the complex structure and metric in $\CC^k$, but the
action is given by the varying group $\G^x_\a = C_\a(x) \G_{\a} C_\a(x)^{-1}$. Clearly 
$\G^x_\a < \UU(k)$ because it preserves the hermitian structure $(\o_0, g_0, J_0)$.
The group $\G^x_\a$ acts on the fiber $\{x\}\x \CC^k$ and 
vary with the point $x \in U_\a$, so the action is not linear on the chart
$U_\a \x \CC^k$.
On the other hand, in the coordinates $(x,\tilde{y})$ the transition functions of the bundle 
are $\UU(k)$-valued as we want.

Now define new coordinates $(x,y')=(x,r(C_\a(x))^{-1} \tilde{y})$ where $r$ is the retraction (\ref{eqn:r(A)}).
The hermitian metric in the new coordinates is the standard metric of $\CC^k$ because it 
was so in the coordinates $(x,\tilde{y})$ and $r(C_\a(x))^{-1} \in \UU(k)$.
So the orbifold almost K\"ahler structure of the fibers in the coordinates $(x,y')$ is given by 
$(\o_0,J_0, g_0)$. However, the isotropy group is the group $\G_\a < \UU(k)$ that we began with.
Indeed, $\G_\a= C_\a(x)^{-1}  \G^x_{\a}  C_\a(x)$ implies, by Lemma \ref{lem:retrac},
that $\G_\a= r(C_\a(x))^{-1}  \G^x_{\a} \, r(C_\a(x))$.
Carrying out this procedure for each coodinate patch, the corresponding transition 
functions are in $\UU(k)$, whereas the isotropy is given by the groups $\G_\a <\UU(k)$.
\end{proof}

\begin{corollary} \label{cor:hermitian-fiber-standard}
We can find adequate trivializations of $\nu_D$ so that the almost K\"ahler structure $(\o_{\nu_D},J_{\nu_D})$ of Proposition \ref{prop:linear-Kahler} defines the standard hermitian metric restricted to the fibers $\CC^k/\G$.
\end{corollary}

\begin{corollary} \label{cor:Gamma}
If $D \subset X$ is a connected homogeneous isotropy submanifold, 
then the normal bundle admits an atlas $\{U_\a \times \CC^k\}$
with the transition functions $A_{\a \b}: U_\a\cap U_\b \to \UU(k)$ and with the group $\G$ fixed.
Actually, the image of $A_{\a \b}$ lies in the normalizer of $\G < \UU(k)$, i.e.\ in the subgroup
of $\UU(k)$ given by $N_{\UU(k)} ( \G) = \{A \in \UU(k) | A \G A^{-1} = \G \}$.
\end{corollary}

\begin{remark}
Therefore, if an homogeneous isotropy submanifold $D \subset X$ has an isotropy group $\G < \UU(k)$ 
with finite normalizer in $\UU(k)$, then its normal bundle $\nu_D$
has constant transition functions $A_{\a \b}$, so the Chern classes $c_k(\nu_D)=0$
for all $k$.
\end{remark}

\section{Gluing the symplectic form in $\nu_D$ to $0$ preserving positivity}

For later purposes, we will need to change the symplectic form $\o_{\nu_D}$ of Proposition \ref{prop:linear-Kahler}. 
Concretely, we need to make $\o_{\nu_D}$ vanish along the fibers, so we can define an smooth form in the underlying space 
of the orbifold $\nu_D$ by pushing forward via the map $q: \nu_D \to |\nu_D|$,
where $|\nu_D|$ is the underlying space. The map $q$ is simply the quotient map in each orbifold chart.
We need to interpolate the symplectic form in $\nu_D$ with $0$ near $D$,
but preserving semi-positivity with respect to an almost complex structure in $\nu_D$.
We need first a lemma that makes this process in $\CC^k$.
Recall that a $2$-form $\o$ is \emph{positive} with respect to an almost complex structure $J$
if $\o(u,Ju)>0$ for all vectors $u \ne 0$; it is \emph{semipositive} if $\o(u,Ju) \ge 0$.

\begin{lemma} \label{lem:extension}
Let $V \subset \CC^k$ open, and $f:V \to \RR$ a smooth function
such that $\tfrac{\ii}{2} \bd \bar \bd f$ is $J_0$-semipositive.
Let $h: \RR \to \RR$ smooth with $h'\ge 0$, $h''\ge 0$, 
and denote $\o=\tfrac{\ii}{2} \bd \bar \bd f$, $\o_h=\tfrac{\ii}{2} \bd \bar \bd (h \circ f)$.

Then the form $\o_h$ is $J_0$-semipositive.
Moreover, $\o_h$ is $J_0$-positive on the subset of $V$ where
$\o$ is $J_0$-positive and $h' \circ f >0$.
\end{lemma}
\begin{proof}
This is a straightforward computation, see for instance \cite[Lemma 31]{MM-R}.
\end{proof}

\begin{remark} \label{rem:bar-omega_0}
Lemma \ref{lem:extension} will be used mainly in $\CC^k$ for the standard symplectic form $\o_0= \frac{\ii}{2} \bd \bar \bd (|z|^2)$
with a suitable choice of function $h$ as follows. 
Choose numbers $0<t_0<t_1<1$ and let $h : \RR \to \RR$ be a function such that
$h(t)=0$ for $t \le  t_0$, $h(t)=t+c$ for $t \ge  t_1$, for some $c\in \RR$, and with $0 \le h' \le 1$ and $0 \le h'' \le \tfrac{2}{t_1-t_0}$.
For instance, take a smooth function $\varrho$ so that $\varrho$ vanishes in $(-\infty,t_0)$, 
equals $1$ in $(t_1,+\infty)$, and $0 \le \varrho' \le \tfrac{2}{t_1-t_0}$; then define $h(t)=\int_{-\infty}^t \varrho$.
In this case we denote $\bar \o_0:=\o_h=\frac{\ii}{2} \bd \bar \bd ( h(|z|^2))$, which is a $J_0$-positive interpolation
between $0$ and the standard symplectic form.
\end{remark}


\begin{remark}
Consider the normal bundle $F \to \nu_D \to D$ endowed with the almost K\"ahler structure $(\o_{\nu_D},J_{\nu_D})$
from Proposition \ref{prop:linear-Kahler}. The splitting 
\begin{equation}\label{eqn:Juan-a}
T \nu_D= TF^{\perp \o_{\nu_D}} \oplus TF = \cHH \oplus \cV \, ,
\end{equation}
satisfies that both $\cHH$ and $\cV$ are almost K\"ahler subbundles. We have a corresponding decomposition of $J_{\nu_D}$ as
\begin{equation}\label{eqn:Juan-b}
J_{\nu_D}=p_{\cHH} \circ J_{\nu_D} + p_{\cV} \circ J_{\nu_D}= J_{\cHH} + J_{\cV} \, ,
\end{equation}
being $p_{\cHH}$ and $p_{\cV}$ the projections onto $\cHH$ and $\cV$ respectively.
Moreover $J_{\nu_D}=J_{\cV}$ in $\cV$ and $J_{\nu_D}=J_{\cHH}$ in $\cHH$.
In particular,
for any $u=u_h+u_v$ with $u_h \in \cHH$ and $u_v \in \cV$ we have $J_{\nu_D}(u)=J_{\cHH}(u_h)+J_{\cV}(u_v)$.
\end{remark}

In the next proposition we modify the simplectic form $\o_{\nu_D}$ of $\nu_D$
from Proposition \ref{prop:linear-Kahler} so as to make it suitable for interpolation later on.

\begin{proposition} \label{prop:linear-kahler-pro}
Let $(X,\o)$ be a symplectic orbifold and $D \subset X$ a HI submanifold.
Consider the normal bundle $F \to \nu_D \to D$ endowed with the almost K\"ahler structure $(\o_{\nu_D},J_{\nu_D})$
from Proposition \ref{prop:linear-Kahler}, the splitting (\ref{eqn:Juan-a})
and the 
decomposition (\ref{eqn:Juan-b}). 
Then we define
$$
\O_{\nu_D} = \pi^* \o_D - \tfrac{1}{4} d J_{\cV} d (|z|^2),
$$
being $|z|$ the height function of the fiber $F \cong \CC^k/\G$. It holds that $\O_{\nu_D}$
is symplectic and compatible with $J_{\nu_D}$ in some neighborhood of the zero section $B_{\d}(D) \subset \nu_D$,
and moreover $\O_{\nu_D}=\o_{\nu_D}=\o$ at all points of the zero section $D \subset \nu_D$.
\end{proposition}

\begin{proof}

Consider trivializations $\nu_D|_{D_\a} \cong D_\a \x \CC^k$ 
as in Corollary \ref{cor:hermitian-fiber-standard} such that $\o_{\nu_D}, J_{\nu_D}$ are the standard symplectic 
form and complex structure restricted to the fibers $\{p\} \x \CC^k$.
Let us denote  the coordinates $(p,z) \in D_\a \x \CC^k$, with $z=x+\ii y$ and $p=(p_1,\dots,p_{2d})$, $d=n-k$.
As the horizontal bundle $\cHH=TF^{\perp \o_{\nu_D}}$ at a point $p \in D \subset \nu_D$ is precisely $T_pD$,
it follows that $(J_{\cV})_p$ maps $\bd_{x_i}$ to $-\bd_{y_i}$ and maps $\bd_{p_i}$ to $0$.
From this we deduce that the expresion of $J_{\nu_D}$ in the coordinates $(p,z)$ is given by
$$
J_{\nu_D}= \sum_{k,j} a_{kj} dp_k \otimes \bd_{x_j} + b_{kj} dp_k \otimes \bd_{y_j} + \sum_{k,l} c_{kl} dp_k \otimes \bd_{p_l}
+ \sum_j dx_j \otimes \bd_{y_j} - dy_j \otimes \bd_{x_j}
$$
where $a_{kj}, b_{kj}$ are functions of $(p,z)$ which vanish at $z=0$. 
Recall that $J_{\cV}=p_{\cV} \circ J_{\nu_D}$, where $p_{\cV}: T \nu_D \to \cV$ is the projection onto the vertical bundle
associated to the decomposition (\ref{eqn:Juan-a}). 
The action of $J_{\cV}$ in forms is by the transpose, so it follows that
\begin{align*}
 &J_{\cV}(dx_i)=-dy_i+ \textstyle \sum_k ( a_{ki} + \textstyle \sum_l c_{kl} d_{il} ) dp_k \, , \\
 &J_{\cV}(dy_i)=dx_i+ \textstyle \sum_k ( b_{ki} + \textstyle \sum_l c_{kl} e_{il} ) dp_k\, ,
\end{align*}
where we have called $d_{il}=dx_i(p_{\cV}(\bd_{p_l}))$, $e_{il}=dy_i(p_{\cV}(\bd_{p_l}))$.
It follows that
\begin{align*}
d J_{\cV} d (|z|^2)= &\, d J_{\cV} d (\textstyle \sum_i x_i^2 + y_i^2)=d J_{\cV}(\textstyle \sum_i 2x_i dx_i + 2y_i dy_i) \\
                   =&\,  d(\textstyle \sum_i 2y_i dx_i - 2x_i dy_i) + \\ & +
                   \textstyle \sum_{k} \left( \textstyle \sum_i 2x_i \left( a_{ki} + \textstyle \sum_l c_{kl} d_{il} \right) 
                   + 2y_i \left( b_{ki} + \textstyle \sum_l c_{kl} e_{il} \right) \right) dp_k \\
                   =& -4\o_0 + \textstyle \sum_k \a_k \wedge dp_k\, ,
\end{align*}
where the $1$-forms $\a_k=O(|z|)$ vanish at $z=0$ (recall that the functions $a_{ki}, b_{ki}, d_{il}, e_{il}$ all vanish at $z=0$), and $\o_0= \sum_i dx_i \wedge dy_i$.
Hence
$$
\O_{\nu_D} = \pi^* \o_D - \tfrac{1}{4} d J_{\cV} d (|z|^2) = \pi^* \o_D + \o_0 + O(|z|),
$$
which proves that $\O_{\nu_D}$ is $J_{\nu_D}$-positive in a neighborhood of the zero section $D$
and moreover $\O_{\nu_D}=\o_{\nu_D}=\o$ at all points of $D$.
\end{proof} 
We therefore have a perturbed almost K\"ahler structure $(\O_{\nu_D},J_{\nu_D})$ in $B_{\d}(D) \subset \nu_D$,
which is more suitable for our purposes.
\begin{remark}
The norm we will use for differential forms will be the operator norm at each point
i.e.\ if $\a \in \O^p(M)$, with $M$ an orbifold or a manifold, its norm at a point $x \in M$ is
$||\a||_x=\max_{|u_i|=1} |\a(u_1,\ldots,u_p)|$.
For a subset $U \subset M$ we set $||\a||_U=\sup_{x \in U} ||\a||_x$.
\end{remark}
In the next proposition we make use 
of $\O_{\nu_D}$ and interpolate it positively with $0$.
Recall the map $q: \nu_D \to |\nu_D|$
from the orbifold $\nu_D$
to its underlying space $|\nu_D|$, which consists of applying
fiberwise the quotient $\CC^k \to \CC^k/\G$. 
Consider the fiber bundle
$\CC^k/\G \to |\nu_D| \to D$, with $\bar \pi: |\nu_D| \to D$ the projection.
Clearly, $(\CC^k-\{0\})/\G \to |\nu_D-D| \to D$ is a smooth fiber bundle, and
the restriction $q:\nu_D - D \to |\nu_D - D|$ is a smooth map between smooth manifolds. 

\begin{proposition} \label{prop:linear-symplectic-vanish-fiber}
Let $(X,\o)$ be a symplectic orbifold and $D$ a HI submanifold.
Consider the total space of the normal bundle $\pi: \nu_D \to D$ with the almost K\"ahler structure
$(\O_{\nu_D},J_{\nu_D})$ of Proposition \ref{prop:linear-kahler-pro}. 
Consider the map $q: \nu_D \to |\nu_D|$ as above.

Then, $\nu_D$ admits an orbifold closed $2$-form $\bar \O_{\nu_D}$ such that,
for $\e>0$ small enough we have:
\begin{itemize}
\item $\bar \O_{\nu_D}=\pi^*\o_D$ in $B_\e(D)$, so the push-forward $q_*(\bar \O_{\nu_D})=\bar \pi^* \o_D$ is a smooth form in the smooth manifold $|\nu_D-D|$.
\item $\bar \O_{\nu_D}$ is $J_{\nu_D}$-positive in $B_{2\e}(D) - B_\e(D)$.
\item $\bar \O_{\nu_D}=\O_{\nu_D}$ in $\nu_D - B_{2\e}(D)$.
 \end{itemize}
Moreover, it is given by the formula
$$
\bar \O_{\nu_D}=\pi^* \o_D - \frac{1}{4} d J_{\cV} d(h(|z|^2)),
$$
with $h$ the function of Remark \ref{rem:bar-omega_0}.
\end{proposition}

\begin{proof}
We take a trivialization $\nu_D|_{D_\a} \cong D_\a \x \CC^k$ with coordinates $(p,z)$, $p=(p_1,\dots,p_{2d})$, $z=x+\ii y$ as in Proposition \ref{prop:linear-kahler-pro}, and compute:
\begin{align*}
d J_{\cV} d (h \circ |z|^2)&=d J_{\cV} (h'(|z|^2) d(|z|^2))=d (h'(|z|^2) J_{\cV} d(|z|^2)) \\
                           &=h''(|z|^2)d(|z|^2) \wedge J_{\cV} d(|z|^2) + h'(|z|^2) dJ_{\cV}d(|z|^2) \\
                           &= h''(|z|^2) \b \wedge J_{\cV} \b + h'(|z|^2) (-4 \o_0 + \textstyle \sum_k \a_k \wedge dp_k),
\end{align*}
where we have denoted $\b=d(|z|^2)=\sum_i 2x_i dx_i + 2y_i dy_i$ and used the expression for $dJ_{\cV}d(|z|^2)$
from the proof of Proposition \ref{prop:linear-kahler-pro}.
Now recall that for a non-zero tangent vector $u \in \cV \subset T \nu_D$ we have 
\begin{equation} \label{eq:beta-positive}
(\b \wedge J_{\cV} \b)(u,J_{\cV}u)= -\b(u)^2-\b(J_{\cV}u)^2 \le 0,
\end{equation}
and as $\b$ only vanishes at $D=\{z=0\}$, we see that $(\b \wedge J_{\cV} \b)(\cdot, J_{\cV}\, \cdot)<0$ outside $D$.
Now we take the above into account and obtain
\begin{align*}
\bar \O_{\nu_D} &=\pi^* \o_D - \tfrac{1}{4} d J_{\cV} d(h(|z|^2)) \\
                &=\pi^* \o_D - \tfrac{1}{4} h''(|z|^2) \b \wedge J_{\cV} \b + h'(|z|^2) \o_0 
                -\tfrac{1}{4} h'(|z|^2) \textstyle \sum_k \a_k \wedge dp_k\, .
\end{align*}
We choose the function $h$ as in Remark \ref{rem:bar-omega_0}
taking $t_0=\e^2$, $t_1=(2\e)^2$, so that it satisfies:
$h(t)=0$ for $t \le \e^2$, 
$h'>0$ and $h''>0$ for $\e^2 < t < (2 \e)^2$,
$h(t)=t+c$ for $t \ge (2 \e)^2$, and
$0 \le h' \le 1$, $0 \le h'' \le \frac{2}{3\e^2}$ .
Now we have cases:
\begin{itemize}
\item For $|z| \le \e$ we have $\bar \O_{\nu_D}=\pi^* \o_D$.
\medskip
\item For $|z| \ge  2 \e$ we have $\bar \O_{\nu_D}=\O_{\nu_D}$.
\medskip
\item For $\e < |z| < 2 \e$ we have $\bar \O_{\nu_D}|_{\cV}= - \tfrac{1}{4} h''(|z|^2) \b \wedge J_{\cV} \b|_{\cV} 
+ h'(|z|^2) \o_0|_{\cV}$ which is clearly $J_{\cV}$-positive by \eqref{eq:beta-positive}.
\medskip
\item For $\e < |z| < 2 \e$ we claim that
\begin{align*}
\bar \O_{\nu_D}|_{\cHH} 
                        & = \pi^* \o_D|_{\cHH} - \tfrac{1}{4} h''(|z|^2) (\b \wedge J_{\cV} \b) |_{\cHH} + h'(|z|^2) \o_0|_{\cHH} 
                -\tfrac{1}{4} h'(|z|^2) \textstyle \sum_k (\a_k \wedge dp_k)|_{\cHH} \\
                        & = \pi^* \o_D|_{\cHH} + O(\e),
\end{align*}
where $O(\e)$ represents a $2$-form defined in $\cHH$ with norm $O(\e)$. 
To see the above bound first note that $\a_k=O(|z|)$ from the proof of Proposition \ref{prop:linear-kahler-pro}; 
secondly,
\begin{align*}
   \tfrac{1}{4} (\b  & \wedge J_{\cV} \b)|_{\cHH} \\
&=  \textstyle \sum_i (x_i dx_i + y_i dy_i ) 
\wedge \left( \textstyle \sum_j x_j(-dy_j + O(|z|)) + 2y_j(dx_j + O(|z|)) \right)|_{\cHH} \\
&=  O(|z|^4),
\end{align*}
because $dx_i|_{\cHH}=O(|z|)$, $dy_i|_{\cHH}=O(|z|)$, which also implies that
$\o_0|_{\cHH} =O(|z|^2)$. Using the properties of $h$ above,
this proves the bound $\bar \O_{\nu_D}|_{\cHH} = \pi^* \o_D|_{\cHH} + O(\e)$.
Now, as $\pi^* \o_D|_{\cHH}$ is $J_{\cHH}$-positive
at all the points of the zero section $D$, it follows easily that for $\e>0$ small enough the $2$-form
$\bar \O_{\nu_D}|_{\cHH}$ is also $J_{\cHH}$-positive in $\{\e < |z| < 2 \e\}$. 

\end{itemize}

Finally, the map $q:\nu_D \to |\nu_D|$ is given in each trivialization
as $q:D_\a \x \CC^k \to D_\a \x \CC^k/\G$, $q(p,z)=(p,[z])$,
so for $0<|z| \le \e$ we have $q_* \bar \O_{\nu_D}=q_* \pi^* \o_D= \bar \pi^* \o_D$,
with $\bar \pi:|\nu_D| \to D$ the projection given in charts by $\bar \pi(p,[z])=p$. 
Hence the push-forward $q_* \bar \O_{\nu_D}$ is a smooth $2$-form in the smooth manifold $|\nu_D-D|$
which equals $\bar \pi^* \o_D$ in $\{0<|z|\le \e\} \subset |\nu_D-D|$.
\end{proof}

\section{Resolution of the normal bundle}

In this section we will use the previous nice structure of the normal bundle $\nu_D$ of an 
$HI$-submanifold $D \subset X$ of a symplectic orbifold $X$,
to construct a symplectic resolution of $\nu_D$.

By Corollary \ref{cor:Gamma}, we fix 
an atlas $\{U_\a\x \CC^k\}$ with $\G < \UU(k)$ acting on the fiber, and with
the transition functions $A_{\a \b}: U_\a  \cap U_\b \to N_{\UU(k)}(\G)$. The group
$G= N_{U(k)}(\G)$ is a closed Lie subgroup of $\UU(k)$ since $\G$ is finite.
In particular $G$ is compact, and acts on $\CC^k/\G$ by matrix multiplication.
Recall that $F_x \cong \CC^k/\G$ is a singular complex variety, hence it admits a constructive
algebraic resolution, see \cite{EV} and \cite{EV_2}.
This resolution has the property that any algebraic action on the singular variety admits a unique
lifting to the resolution. 

\begin{theorem}[{\cite[Prop.\ 7.6.2]{EV_2}}] \label{th:orlando} 
Let $X \subset W$ be a subscheme of finite type of a smooth scheme $W$, with $X$ reduced,
and $\theta \in \Aut(W)$ an algebraic automorphism of $X$.
Let $b: \tilde{X} \to X$ be the constructive resolution of singularities.
Then $\theta: X \to X$ lifts uniquely to an isomorphism $\tilde{\theta}: \tilde{X} \to \tilde{X}$
of the constructive resolution of singularities $\tilde{X}$ of $X$ such that $b \circ \tilde{\theta}= \theta \circ b$.
\end{theorem}

Note that the uniqueness of the lifting follows immediately from the existence 
because any two liftings have to coincide in the Zariski open set where $b: \tilde{X} \to X$ is an isomorphism.

The compact group $G=N_{\UU(k)}(\G)<\UU(k)$ has a complexification $G^c<\GL(k,\CC)$ which is
an algebraic group.
We claim that $G^c<N_{\GL(k,\CC)}(\G)$. The normalizer $N_{\GL(k,\CC)}(\G)<\GL(k,\CC)$ is a complex Lie
group that contains $G$, hence it contains $G^c$, which is its Zariski closure. 
Thus the group $G^c$ acts naturally on $F=\CC^{k}/\G$ by matrix multiplication, 
i.e.\ $A \cdot [u]=[Au]$ for $A \in G^c$. 
Here the bracket stands for the equivalence class of $u \in \CC^k$ in the quotient $\CC^k/\G$.
For $A\in G^c$, this is well defined because if $[u]=[u']$ then there exists $\g \in \G$ with
$u=\g u'$ and hence $Au=A \g u'= \g' A u'$ for some $\g' \in \G$, since $A \in N_{\GL(k,\CC)}( \G)$.

\begin{proposition} \label{prop:emb}
The fiber $F= \CC^{k}/\G$ and its constructive resolution $\tilde{F}$ are quasi-projective varieties.
\end{proposition}

\begin{proof}
Since $\G < \UU(k)$ is a finite group, the quotient $F= \CC^{k}/\G$
is an affine variety, i.e.\ there is an embeding $\imath:F \to \CC^N$ for some $N \in \NN$. 
Indeed $\CC [x_1,\dots ,x_k]^{\G} \subset \CC [x_1,\dots ,x_k]$, the $\CC$-algebra of polynomials 
invariant by the action of $\G$, is a finitely generated $\CC$-algebra, say
$\CC [x_1,\dots ,x_k]^{\G} = \CC [f_1, \dots , f_N ]$ for some $f_j \in \CC[x_1, \dots ,x_k]$.
Defining $\imath: \CC^k/\G \to \CC^N$, $\imath([(x_1,\dots, x_k)])= (f_1(x), \dots , f_N(x))$, we 
have an embedding of $F$ into $\CC^N$. 
This proves that $F$ is an affine variety, hence it is quasi-projective. 
We can use the model $\imath(F) \subset \CC^N$ to perform the resolution of singularities.
The resolution $\tilde{F}$ of $\imath(F)$ is obtained via a finite numbers of blow-ups 
starting from $\CC^N$ so $\tilde{F}$ is quasi-projective.
\end{proof}

Select an embedding $\imath: F=\CC^k/\G \to \CC^N$ as in Proposition \ref{prop:emb}. 
Let $\tilde{F}$ be the constructive resolution of the algebraic variety $\imath (F) \subset \CC^N$.
The action $G^c \x F\to F$, $(g,y)\mapsto g y$, is an algebraic map.
There is a well-defined map $G^c \x \tilde F\to G^c \x \tilde F, (g,y) \mapsto (g, g \cdot y)$,
by Theorem \ref{th:orlando}. 
This is a bijection between smooth algebraic varieties, and it is algebraic
on the Zariski dense open subset $G^c \x \tilde{F}  -  G^c \x Z$, where $Z$ is the exceptional locus.
In particular it is continuous.
Therefore it is algebraic everywhere. This implies
that the map $G^c \to \Aut(\tilde F)$ is also algebraic, 
in particular the map $G \to \Aut(\tilde F)$ is smooth.

Let $b: \tilde{F} \to F$ be the blow-up map, and denote by $Z=b^{-1}(0)$ the exceptional divisor.
For the bundle $\nu_D$, each transition matrix $A_{\a \b}(x) \in G<\UU(k)$
has a corresponding unique lifting $B_{\a \b}(x): \tilde{F} \to \tilde{F}$
which satisfies $b(B_{\a \b}(x) y)= A_{\a \b}(x) (b(y))$, for each $y \in \tilde{F}$, i.e.\ 
$b \circ B_{\a \b}(x)= A_{\a \b}(x) \circ b$.
The maps $B_{\a \b}(x)$ depend smoothly on $x$, since $A_{\a \b}(x)$ depend smoothly on $x$
and the map $G \to \Aut(\tilde F)$ is smooth.

\begin{proposition}
The maps $B_{\a \b}(x)$ for $x \in U_\a \cap U_\b$ are the transition
functions of a smooth fiber bundle $\tilde\nu_D\to D$ with $\tilde{F}$ as fiber.

There is a map $b: \tilde\nu_D \to \nu_D$ which is a diffeomorphism outside
the subbundle $E \to D$ whose fiber is the exceptional locus  $Z\subset \tilde{F}$.
\end{proposition}

\begin{proof}
We only need to check the cocycle condition.
In a triple intersection we know that $A_{\a \b} \circ A_{\d \a} \circ A_{\b \d}= \Id_F$.
Since lifting respects composition and the identity lifts to the identity, we have that
$B_{\a \b} \circ B_{\d \a} \circ B_{\b \d}= \Id_{\tilde{F}}$, as required. 
\end{proof}

We call ${b}$ the blow-up map, because it is induced on each fiber by the 
blow-up map $b: \tilde{F} \to F$.

The next step consists on constructing a symplectic form on 
the resolution $\tilde{F}$ of the complex variety $F = \CC^k/\G$, with $\G < \UU(k)$ as above.
Here, $F\cong F_x$ is diffeomorphic to the orbifold normal space $(T_xD^{\perp})/\G$ 
of the HI submanifold $D \subset X$.
Since $D$ does not intersect any other isotropy submanifold of the orbifold $X$, we see that $0 \in \CC^k$
is the only fixed point of the action of the group $\G<\UU(k)$.
Hence the singular locus of $F$ reduces to the point $[0] \in F= \CC^k/\G$. 
The exceptional locus is $Z= b^{-1}([0]) \subset \tilde{F}$, and consists of a finite union of irreducible
components $Z_j$ which are divisors intersecting transversally.

\begin{proposition} \label{prop:resolution-fiber-Kahler}
The resolution $\tilde{F}$ of $F= \CC^k/\G$ admits a K\"ahler structure $(\o_{\tilde{F}}, J_{\tilde{F}}, g_{\tilde{F}})$ 
which is invariant by the action of $G=N_{\UU(k)} (\G)$ on $\tilde{F}$.
\end{proposition}

\begin{proof}
By Proposition \ref{prop:emb}, $\tilde{F}$ is a quasi-projective variety, so it is a complex submanifold of
some $\CC P^N$ for $N$ high enough. Consider $(\CC P^N, \o_{FS}, J, g_{FS})$
the standard K\"ahler structure on $\CC P^N$, where
$\o_{FS}$ is the Fubini-Study K\"ahler form.
The restriction of $(\o_{FS}, J, g_{FS})$ to $\tilde{F}$ defines a K\"ahler structure $(\o_1, J_{\tilde{F}}, g_1)$ on $\tilde{F}$,
where $J_{\tilde F}$ is the given complex structure on $\tilde F$.

The complex structure $J_{\tilde{F}}$ is preserved by the transition functions $B_{\a \b}(x)$
because they act on $\tilde{F}$ as biholomorphisms.
But the symplectic structure $\o_1$ may not be preserved, so we need to make an average. As 
$G$ is compact, we 
put on $G$ any right-invariant Riemannian metric and call $\mu$ the measure induced by this metric.
Let 
$$
\o_{\tilde{F}}= \frac{1}{\mu(G)} \int_{G} h^*\o_1 d \mu(h) \in \O^2(\tilde{F}).
$$
We claim that $\o_{\tilde{F}}$ is a symplectic form invariant by the action of $G$ on $\tilde{F}$.
For the invariance, take $g \in G$ and compute
 \begin{align*}
 g^* \o_{\tilde{F}} &= \frac{1}{\mu(G)} \int_{G} g^*(h^* \o_1) d \mu(h) \\ 
 &= \frac{1}{\mu(G)} \int_{G} (hg)^*(\o_1) d \mu(h) =\frac{1}{\mu(G)} \int_{G} k^*\o_1 d \mu(k)=  \o_{\tilde{F}},
 \end{align*}
where we have made the change of variables $hg=k$, and $d \mu(h)= d \mu(k)$  since translations are isometries.
The closeness is clear as $d \o_{\tilde{F}}= \frac{1}{\mu(G)} \int_{G} d (h^*\o_1) d \mu(h) =0$.
Finally, let us see that $\o_{\tilde F}$ is a K\"ahler form. As $\omega_1(u,v)=g_1(u,-Jv)$, we have that
$h^*\omega_1(u,v)=h^*g_1(u,-Jv)$, and hence $\omega_{\tilde F}(u,v)=g_{\tilde F}(u,-Jv)$, where
 $g_{\tilde{F}}= \frac{1}{\mu(G)} \int_{G} h^*g_1 d \mu(h)$ is a $G$-invariant Riemannian metric. Moreover
$g_{\tilde F}(Ju,Jv)=g_{\tilde F}(u,v)$. 
This gives a K\"ahler structure $(\o_{\tilde{F}}, J_{\tilde{F}}, g_{\tilde{F}})$ on $\tilde{F}$ invariant by the
action of the group $G$, as desired.
\end{proof}

Let $b:\tilde F\to F$ be blow-up map, $Z=b^{-1}(0) \subset \tilde F$ the exceptional divisor.
So $b:\tilde F-Z \to  F-\{0\}$ is a biholomorphism. 
We now modify the K\"ahler form $\o_{\tilde{F}}$ in the complement of a neighbourhood of $Z$ 
to make it agree with $b^* \o_F$, the K\"ahler form on the fiber $F=\CC^k/\G$. 
Let us introduce some useful notation.
Given a ball $B_\e(0) \subset \CC^n$, denote $B_\e(Z)=b^{-1}(B_\e(0))$.
Clearly $\{ B_\e(Z)| \e>0 \}$ gives a basis of neighborhoods for $Z$ in $\tilde F$.
Take $q: \CC^k \to \CC^k/\G=F$ the quotient map and consider the orbifold symplectic 
form on $F$ which is given by $\o_F=q_* \o_0 \in \O^2(F-\{0\})$. 
Note that $q_* \o_0$ is not defined at the singularity $0 \in F$,
and consequently $b^* q_* \o_0$ is not defined at $Z=b^{-1}(0)$.
To handle this issue, we interpolate $\o_0$ with $0$ near the origin of $\CC^k$ using Lemma \ref{lem:extension}.
Select numbers $t_1>t_0>0$ and take a function $h: [0,\infty) \to \RR$ 
as in Remark \ref{rem:bar-omega_0}.
Consider the closed $2$-form 
$$
\bar{\o}_0= \tfrac{\ii}{2} \bd \bar \bd (h \circ |z|^2) \, .
$$
By Lemma \ref{lem:extension}, $\bar{\o}_0$ is $J_0$-positive
in the set $\{z \in \CC^k |\, |z| > t_0\}$. It vanishes in $\{z \in \CC^k |\, |z| \le t_0\}$,
and moreover $\bar{\o}_0=\o_0$ in the set $\{z \in \CC^k |\, |z| \ge t_1\}$.
Since $q_*(\bar{\o}_0)$ vanishes in a neighborhood of 
the singular point $0 \in F$, we have a smooth form 
$b^* q_* \bar{\o}_0 \in \O^2(\tilde F)$.
As $b: \tilde F \to F$ is holomorphic, the form $b^* q_* \bar{\o}_0$ is $J_{\tilde F}$-positive
on $\tilde F - \bar{B}_{t_0}(Z)=b^{-1}(\{|z| > t_0\})$. Also it vanishes on $B_{t_0}(Z)$,
and $b^* q_* \bar{\o}_0=b^* \o_F$ outside $B_{t_1}(Z)$.
In the following proposition we interpolate the symplectic form $\o_{\tilde F}$ 
constructed in Proposition \ref{prop:resolution-fiber-Kahler} and $b^* q_* \bar{\o}_0$.

\begin{proposition} \label{prop:symplectic-form-tilde F}
Given positive numbers $t_0<t_1$
there exists a K\"ahler form $\O_{\tilde{F}}$ in $\tilde F$ such that:
\begin{itemize}
\item It coincides with the form $b^*(\o_F)$ outside $B_{t_1}(Z)$,
being $\o_F=q_*(\o_0)$ the symplectic form on $F-\{0\}$.
\item It is invariant by the transition functions $B_{\a \b}$ of the bundle $E$.
\end{itemize}
\end{proposition}

\begin{proof}
Consider $(\o_{\tilde F}, J_{\tilde F})$ the K\"ahler structure 
on $\tilde F$ constructed in Proposition \ref{prop:resolution-fiber-Kahler}.
Since $\tilde F - Z \cong \CC^k /\G - \{0\}$, it follows that $H^2(\tilde F-Z,\RR)=0$, 
so $\o_{\tilde F}-b^* q_* \bar{\o}_0= d \eta$ for some $1$-form $\eta \in \O^1(\tilde F - Z)$. 
Select a number $t_2 \in (t_0, t_1)$
and take a positive smooth function $\rho(t)$ which vanishes on $\{t \ge t_1\}$ and equals $1$
on $\{t \le t_2\}$. Consider $\O_{\tilde F, \, \l}= b^* q_* \bar{\o}_0 + \l d( \rho \eta)$ for $\l>0$
to be chosen later. Clearly, we have
$$
\O_{\tilde F, \, \l} =
\begin{cases}
\l \o_{\tilde F} + (1-\l) b^* q_* \bar{\o}_0  , & \text{on } b^{-1}(\{|z| \le t_2\}) \\
\l \rho \o_{\tilde F} + (1- \l \rho) b^* q_* \bar{\o}_0 + \l d \rho \wedge \eta   , & \text{on } b^{-1}(\{t_2 \le |z| \le t_1\}) \\
b^* q_* \bar{\o}_0= b^* \o_F  , & \text{on } b^{-1}(\{|z| \ge t_1\}) \, .
\end{cases}
$$
Now
\begin{itemize}
\item On $b^{-1}(\{|z| \le t_0\})$ we have that $b^* q_* \bar{\o}_0=0$, so $\O_{\tilde F, \, \l}=\l \o_{\tilde F}$ is $J_{\tilde F}$-positive.

\item On $b^{-1}(\{t_0 \le |z| \le t_2\})$ the forms $\o_{\tilde F}$ and $b^* q_* \bar{\o}_0$ are $J_{\tilde F}$-positive and $J_{\tilde F}$-semipositive respectively. This implies that $\O_{\tilde F, \, \l}$ is $J_{\tilde F}$-positive.

\item On $b^{-1}(\{t_2 \le |z| \le t_1\})$, the form $b^* q_* \bar{\o}_0$ is $J_{\tilde F}$-positive
so there exists a constant $c>0$ so that $\tilde \o_1(u,J_{\tilde F}u) \ge  c |u|^2$
for all tangent vectors $u$. On the other hand $|(d \rho \wedge \eta) (u,v)| \le M |u| |v|$ for some $M>0$ independent of $\l$
and all tangent vectors $u,v$.
Hence, choosing $\l>0$ small enough, we can ensure that 
\begin{align*}
\O_{\tilde F, \, \l}(u, J_{\tilde F}u) =(\l \rho \o_{\tilde F} + (1- \l \rho) \tilde \o_1 + \l d \rho \wedge \eta)(u, J_{\tilde F}u) \ge \frac{c}{2} |u|^2 \, .
\end{align*}

\item On $b^{-1}(\{|z| \ge t_1\})$, $\O_{\tilde F, \, \l}=b^* q_* \bar{\o}_0= b^* \o_F$ is clearly
$J_{\tilde F}$-positive.
\end{itemize}

It only remains to make the form $\O_{\tilde F, \, \l}$ invariant by the transition functions $B_{\a \b}$ of the bundle $E$.
As $B_{\a \b}$ take values in the compact group $G=N_{\UU(k)}(\G)$, we can make an average 
as in Proposition \ref{prop:resolution-fiber-Kahler} and the result follows. Note that the $J_{\tilde F}$-positivity is
preserved after this average as the functions $B_{\a \b}$ act on $\tilde F$ by biholomorphisms.
\end{proof}

The proposition above shows that we can construct a symplectic form $\O_{\tilde{F}}$ on the fiber
$\tilde{F}$ of $\tilde{\nu}_D$ which coincides with the symplectic form of $F$
outside a neighborhood of the exceptional set $Z$. This gives a symplectic resolution of the normal fibers $F$
of $\nu_D$. Now we will globalize the construction to obtain a symplectic form in some small neighborhood of the
exceptional locus $E=b^{-1}(D)$ of $\tilde\nu_D$. 
Note that the restriction $b|_E:E \to D$ is a fibre subbundle of $\tilde \nu_D$, whose fiber is $Z$.


\begin{remark} \label{rem:cohom-obs}
The question of whether a bundle with symplectic fibers over a symplectic base space admits a symplectic form
defined on the total space of the bundle is not entirely trivial and there are some topological obstructions \cite{GLSW}.
For instance, consider the Hopf fibration $S^1 \to S^3 \to S^2$ and
multiply by $S^1$ to get a torus bundle $S^1 \x S^1 \to S^3 \x S^1 \to S^2$. Both base and fiber are symplectic,
however the total space has trivial second cohomology so it is not symplectic.
\end{remark}

The first thing that we need is to find a cohomology class $[\eta]$
on the manifold $\tilde\nu_D$ that 
restricts to the cohomology class $[\O_{\tilde{F}}]$.

\begin{proposition}
The homology group $H_{2k-2}(\tilde{F},\ZZ)$ of $\tilde{F}$ is freely generated by the exceptional divisors 
$Z_j$, $j=1,\ldots, l$ (the irreducible components of $Z\subset \tilde{F}$).
In other words $H_{2k-2}(\tilde{F},\ZZ) = \bigoplus_{j=1}^l \ZZ \langle Z_j \rangle$.
\end{proposition}

\begin{proof}
The exceptional locus $Z$ of the constructive resolution of singularities of \cite{EV} is a tree
of exceptional divisors $Z_j$ with normal crossings. 
By transversality, $Z_i\cap Z_j$ for $i\neq j$ is of complex dimension $\leq (k-2)$, 
hence of real dimension $\leq (2k-4)$. So  
 \begin{align*}
  H_{2k-2}(Z) &= H_{2k-2}\left(Z/(\cup_{i\neq j} (Z_i\cap Z_j))\right)=H_{2k-2}\left(
  \bigvee\nolimits{}_{j=1}^l Z_j/(\cup_{i\neq j} (Z_i\cap Z_j))\right) \\
 &\cong \bigoplus_{j=1}^l H_{2k-2}\left(Z_j/(\cup_{i\neq j} (Z_i\cap Z_j))\right) =
 \bigoplus_{j=1}^l H_{2k-2}(Z_j) = \bigoplus_{j=1}^l \ZZ\la Z_j \ra.
 \end{align*}
There is a deformation retract from $\tilde{F}$ to $Z$ induced by lifting the radial vector field $-r\frac{\bd}{\bd r}$ 
from $F=\CC^k/\G$ to $b:\tilde{F}\to F$. Therefore $H_{2k-2}(\tilde{F}) =H_{2k-2}(Z)=\bigoplus_{j=1}^l \ZZ\la Z_j \ra$,
as required.
\end{proof}

This proposition means that in the bundle $\tilde{F} \to \tilde\nu_D \to D$ there is a canonical
unordered basis for $H_{2k-2}(\tilde{F})$ at the level of chains, namely the set of exceptional divisors.
Note that for each ordering of the exceptional divisors $Z_j$, we have a basis of $H_{2k-2}(\tilde{F})$,
but the transition functions $B_{\a \b}(x): \tilde{F} \to \tilde{F}$ 
induce a permutation on this basis, so it is the (unordered) set $\{Z_1, \ldots, Z_l\}$ what is preserved.


Poincar\'e duality for $\tilde F$ gives an isomorphism
$$ 
PD: H_c^2(\tilde F,\RR) \stackrel{\cong}{\longrightarrow} H_{2k-2}(\tilde F,\RR).
$$
Note that $H_c^2(\tilde F,\RR)\cong H^2(\tilde F,\RR)$. To see this, consider the 
radial function $r:\tilde F\to [0,\infty)$ given by $r(y)=|b(y)|$, and the open sets 
$$
A_R=B_R(Z)=\{y\in \tilde F| r(y) < R\}=b^{-1}(B_R(0)/\G) \subset \tilde F
$$
for each $R>0$. Then it follows that
$$
H^2_c(\tilde F,\RR) \cong H^2_c(A_R,\RR) \cong H^2( \bar{A}_R,\bd A_R,\RR) \cong
H^2( \bar A_R,\RR) \cong H^2(\tilde F,\RR) \, .
$$
The first isomorphism is obtained using the Mayer-Vietoris sequence for cohomology with compact support.
The third comes from the exact sequence for relative cohomology, using that $\bd A_R \cong S^{2k-1}/\G$ has 
$H^i(\bd A_R,\RR)=H^i(S^{2k-1},\RR)^{\G} \cong 0$ for all $i \ne 2k-1$. 
The fourth is proved either using the Mayer-Vietoris sequence for ordinary cohomology, or using the fact that
$\tilde F$ deformation retracts onto $\bar A_R$.

\section{Symplectic form on the resolution of the normal bundle}

First we deal with the cohomological obstruction mentioned in Remark \ref{rem:cohom-obs}.


\begin{proposition}
Let $\tilde{F} \to {\tilde{\nu}_D} \xrightarrow{\tilde \pi} D$ be as before, 
with $(\tilde{F},\O_{\tilde{F}}, J_{\tilde{F}})$ the K\"ahler structure on $\tilde{F}$.
There exists a cohomology class $a \in H^2({\tilde{\nu}_D},\RR)$ 
whose restriction to each fiber is $[\O_{\tilde{F}}] \in H^2(\tilde F,\RR)$.
\end{proposition}

\begin{proof}
Consider the atlas of the bundle ${\tilde{\nu}_D}$ consisting of charts
$\phi_\a: U_\a \times \tilde{F}\to \tilde V_\a  \subset {\tilde{\nu}_D}$, and with
change of trivializations
$B_{\a \b}: U_\a \cap U_\b \to \Sympl(\tilde{F}, \O_{\tilde F})$.
We refine the open cover given by the $U_\a \subset D$ in such a way that 
there exists a smooth map $T_\a: [0,1]^{2n-2k} \to U_\a$ with image $Q_\a \subset U_{\a}$, so
that the simplices $Q_\a$ form a triangulation of $D$.
As $D$ is compact and symplectic, it is an oriented manifold of dimension $2n-2k$. Let
$[D]\in H_{2n-2k}(D)$ denote its fundamental class, which can be defined by the 
chain $\sum_\a Q_\a\in C_{2n-2k}(D)$.

On the other hand, consider the cohomology class $[\O_{\tilde F}] \in H^2(\tilde{F},\RR)$.
We saw before that $H^2(\tilde{F},\RR) \cong H^2_c(\tilde{F},\RR)$,
and by Poincar\'e duality $H^2_c(\tilde{F},\RR) \cong H_{2k-2}(\tilde{F},\RR)$.
Choose a basis $\{ Z_1 ,\dots ,Z_l\}$ of exceptional divisors of $H_{2k-2}(\tilde F)$.
There exists unique real numbers $ a_i \in \RR$ so that 
$\PD [\O_{\tilde F}]=\sum_{i=1}^l a_i [Z_i]$.
For each trivialization $\phi_\a: U_\a \times \tilde{F} \to \tilde V_\a \subset {\tilde{\nu}_D}$,
consider the chain 
 $$
 A_\a= \sum_{i=1}^l  a_i \phi_\a( Q_\a \x Z_i ) \in C_{2n-2}({\tilde{\nu}_D}).
 $$
We claim that the chain $A= \sum_\a A_\a$ is closed,  
so it defines a homology class $[A] \in H_{2n-2}({\tilde{\nu}_D})$.
We have  
 \begin{equation}\label{eqn:bdA}
 \bd A= \sum \bd A_\a=\sum_\a \sum_i  a_i \phi_\a (\bd Q_\a \x Z_i).
 \end{equation}
If $x\in \bd Q_\a \cap \bd Q_\b \subset U_\a \cap U_\b$, the 
transition function $g=B_{\a \b}(x):\tilde{F} \to \tilde{F}$
is a symplectomorphism of $(\tilde{F}, \O_{\tilde F})$, hence 
it preserves the homology class $\PD([\O_{\tilde F}])= \sum_{i=1}^l a_i [Z_i]$.
On the other hand, $g$ permutes the exceptional divisors $Z_i$. 
But if $g(Z_{i_1})=Z_{i_2}$ then the corresponding coefficients in $[\O_{\tilde F}]$ are the same, 
i.e.\ $a_{i_1}= a_{i_2}$. This follows from the equality
$\PD [\O_{\tilde F}]= \sum_{i=1}^l a_i [Z_i] = (g)_*(\PD [\O_{\tilde F}])= \sum_{i=1}^l a_i [g(Z_i)]$,
by looking on both sides at the coefficient of $[Z_{i_2}]$.
Therefore, if $g(Z_{i_1})=Z_{i_2}$ then 
 \begin{equation}\label{eqn:bdA2}
 a_{i_1} \phi_\a \left( T \times Z_{i_1} \right) + a_{i_2} \phi_\b \left( T \times Z_{i_2} \right) = 0 
  \in C_{2n-3}({\tilde{\nu}_D}),
 \end{equation}
where $T\subset \bd Q_\a\cap \bd Q_\b$ is a $(2n-3)$-simplex that is common to the boundary of both
$Q_\a$ and $Q_\b$. Note that we are taking into account that the orientations of $T$ induced 
by $Q_\a$ and $Q_\b$ are opposite. Plugging (\ref{eqn:bdA2}) into (\ref{eqn:bdA}), we get that
$\bd A=0$.

Hence $A \in H_{2n-2}({\tilde{\nu}_D})$ determines via Poincar\'e duality a unique 
${a}=[\eta] \in H^2_c({\tilde{\nu}_D},\RR)$
so that $\PD({a})=A$. The relation between ${a}=[\eta]$ and $A$ is given by the equality 
$\int_{\tilde{\nu}_D}  \eta \wedge \b = \int_A \b$, for all $[\b] \in H^{2n-2}({\tilde{\nu}_D})$.
To see that the cohomology class $[\eta]$ restricts to $[\O_{\tilde F}]$ over each fiber $\tilde{F}$,
we need to check that
$\int_{\tilde{F}} \eta \wedge \g = \int_{\tilde{F}} \O_{\tilde F} \wedge \g$ for all $[\g] \in H^{2k-2}(\tilde{F})$.
For this, take any $x \in D$ with fiber $\tilde{F}_x \subset {\tilde{\nu}_D}$, and some $Q_{\a}$ containing $x$.
Take any $[\g] \in H^{2k-2}(\tilde{F}_x)$. 
Consider a bump $2(n-k)$-form $\nu \in \O^{2n-2k}(D)$ with support contained in $Q_{\a}$ 
and $\int_D\nu =1$. 
Then $\tilde \pi^*\nu$ has support in $Q_{\a} \times \tilde{F}$ and so
\begin{align*}
  \int_{\tilde{F}_x} \eta \wedge \g &= \int_{Q_\a \x \tilde F_x} \eta \wedge \g \wedge \tilde \pi^*\nu 
= \int_{\tilde{\nu}_D} \eta \wedge \g \wedge \tilde \pi^*\nu = \int_A \g \wedge \tilde \pi^*\nu \\
 &= \int_{A \cap (Q_\a\x \tilde F)} \g \wedge \tilde \pi^*\nu 
 = \sum_i a_i \int_{Q_\a\x Z_i}\g \wedge \tilde \pi^*\nu = \sum_i a_i \int_{Z_i} \g 
 = \int_{\tilde{F}_x} \O_{\tilde F} \wedge \g.
\end{align*}
\end{proof}

In \cite{Th2} it is given a construction of a symplectic form on the total space of a fiber bundle with 
symplectic base and compact symplectic fibers, once we know the existence of a cohomology
class that restricts to the cohomology class of the sympletic form on the fibers. We have to
do a slight extension to a case with non-compact symplectic fiber. 

%
%

\begin{definition} \label{def:proper-symp-bun}
Let $B$ be a compact manifold, and $(N,\o_N)$ a
(possibly non-compact) symplectic manifold with a proper \emph{height} function $H:N\to [0,\infty)$.
A \emph{proper symplectic bundle} is a fiber bundle $N\to M \to B$ such  that the transition functions 
take values in $\Sympl(N,\o_N,H)=\{ f:N\to N| f^*\o_N=\o_N, H\circ f=H\}$. 
\end{definition}

If $N\to M\to B$ is a proper symplectic bundle, then the height function $H$ defines
a smooth proper function $H_M:M\to [0,\infty)$. 
For $R>0$, we introduce the sets $M_R=H_M^{-1}([0,R])\subset M$ and $N_R=H^{-1}([0,R])\subset N$.
Then $N_R$ and $M_R$ are compact and $N_R\to M_R\to B$ is a fibre bundle. If $R>0$ is a regular value of $H$,
then $(N_R,\o_R)$ is a symplectic manifold with boundary, so $N_R\to M_R\to B$ is a compact symplectic
bundle.

\begin{proposition} \label{prop:symplectic-bundle}
Let $N \to M \stackrel{\pi}{\longrightarrow} B$ be a proper symplectic bundle, where the base space $(B,\o_B)$ 
is a compact symplectic manifold, $(N,\o_N)$ is a symplectic manifold with height function $H:N\to [0,\infty)$.
Suppose that there exists a cohomology class $e \in H^2(M,\RR)$ which restricts to $[\o_N]$ on every fiber.
Fix $R>0$. 
Then there exists a closed $2$-form $\o_{M,K} \in \O^2(M)$ which is non-degenerate 
on $M_R \subset M$, so that $\o_{M,K}$ restricts to $\o_N$ on every fiber $N_x=\pi^{-1}(x) \subset M$.
\end{proposition}

\begin{proof}
Take $e=[\eta]$ with $\eta \in \O^2(M)$ a representative of the class $e$.
Take $U_{\a}$ a good cover of $B$ so that
$\phi_\a: U_\a \times N \to V_\a \subset M$ are trivialisations of the bundle $M$,
and the transition functions $g_{\a \b}:U_\a \cap U_\b \to \Sympl(N, \o_N,H)$.
On each trivialisation the (locally defined) 
vertical projection $q_\a: U_\a \times N \to N$ induces an isomorphism in cohomology, 
hence $(\phi^{-1}_{\a})^* q_\a^* \o_N - \eta|_{V_\a}= d \theta_\a$ 
for some $1$-form $\theta_\a \in \O^1(V_\a)$.
Take a partition of unity $\rho_\a$ subordinated to the open cover $U_\a$ of $B$ and
define
 \begin{equation}\label{eqn:oM}
 \o_{M,K}= K \pi^*(\o_B) + \eta + \sum_{\a} d((\pi^* \rho_\a) \theta_\a),
 \end{equation}
for a real number $K>0$ to be chosen later.
We claim that $\o_{M,K}$ is symplectic in $M_R \subset M$ if $K>0$ is large enough. 
The form $\o_{M,K}$ is clearly closed. We rewrite it as
 \begin{align*}
 \o_{M,K} &= K \pi^* \o_B + \eta + \sum_\a (\pi^*d \rho_\a) \wedge \theta_\a + 
\sum_\a (\pi^* \rho_\a) \wedge ((\phi^{-1}_{\a})^* q_\a^* \o_N - \eta)  \\
 &= K \pi^* \o_B + \sum_\a (\pi^*d \rho_\a ) \wedge \theta_\a 
 + \sum_\a (\pi^* \rho_\a ) (\phi^{-1}_{\a})^* q_\a^* \o_N = K \pi^* \o_B + \mu   \, .
\end{align*}
On a fiber $N_x=\pi^{-1}(x)$, we have 
 $$
 (\o_{M,K})|_{N_x}= \mu|_{N_x} = \sum_\a \rho_\a(x) (\phi^{-1}_{\a})^* q_\a^* \o_N=
\sum_\a \rho_\a(x) \o_N=\o_N,
 $$ 
since all $\phi_\a:\{x\}\x N\to N_x$ are symplectomorphims.
We are using here that the transition functions of the bundle are symplectomorphisms of $(N,\o_N)$.


To see that $\o_{M,K}$ is non-degenerate on $M_R$, 
take a non-zero vector $u \in T_yM$ and let us see that there exists another vector $u'$
such that $\o_{M,K}(u,u') \ne 0$.
If $u \in T_yN_{\pi(y)}$ lies in the tangent space to the fiber, then the existence of $u'$
with $\o_{M,K}(u,u') \ne 0$ is clear since by construction
$\o_{M,K}|_{N_{\pi(y)}}=\mu|_{N_{\pi(y)}}=\o_N$ is symplectic. 

Let us consider the distribution 
$$
W=TN^{\perp \mu}=\{ v \in TM \, | \,\, \mu(v,\cdot)|_{TN}=0\} \subset TM \, .
$$
Since $\mu|_{TN}$ is non degenerate, we have a direct sum $TN \oplus W = TM$.
Moreover, since $\ker \pi_*=TN$, we have that $W=TN^{\perp \mu}= TN^{\perp \o_{M,K}}$, and
by compactness of $M_R$ it follows the existence of constants $c_1, c_2 >0$ such that
$c_1 |w| \le |\pi_*(w)| \le c_2 |w|$ for all $w \in T W|_{M_R}$.
Also, as $\o_B$ is a symplectic form on $B$, there exists a constant $c_0>0$
such that, given a vector $b \in TB$, there exists another vector $b' \in TB$ with
$\o_B(b,b') \ge c_0 |b|\, |b'|$.

Now take $u \in TM_R$, and write $u=n + w$ with $n \in TN$, $w \in W$.
We write $u'=n' + w'$, so we have
\begin{align*}
\o_{M,K}(u,u') & = (K \pi^* \o_B )(n+w,n'+w') + \mu(n+w,n'+w') \\
 & = (K \pi^* \o_B )(w,w') + \mu(n,n') + \mu(w,w') \, .
\end{align*}
Note that $\mu(w,n')=\mu(w',n)=0$ by the definition of $W$.

In case $w=0$ then $u=n \in TN$ and then we know that we can take $u'=n'$ also in $TN$
since $\mu|_{TN}$ is symplectic.
If $w \ne 0$ then we take a vector $u'=w' \in W$ with $|w'|=1$
and such that $\o_B(\pi_* w, \cdot)$ restricted to the sphere 
attains its maximun in the direction of $\pi_*(w')$; then
\begin{align*}
\o_{M,K}(u,u') & \ge K | \pi^* \o_B (w,w')| - || \mu || \cdot |w| \, |w'| \\
 & \ge K c_0 |\pi_* w| |\pi_* w'| - || \mu || \cdot |w|\,  |w'|  \\
 & \ge K c_0 c^2_1 |w| |w'| - || \mu || \cdot |w| \, |w'| \\
 & = |w|( K c^2_1 c_0  - || \mu || ) >0 \, ,
\end{align*}
as long as we take $K > \frac{|| \mu ||}{c_0 c^2_1}$.
\end{proof}

Applying Proposition \ref{prop:symplectic-bundle} to the symplectic bundle $\tilde{F}\to \tilde\nu_D\to D$
with symplectic fiber $(\tilde F,\O_{\tilde F})$ and height function given by $H(y)=|b(y)|$ for $y\in \tilde F=
\CC^k/\Gamma$, we have the following.

\begin{theorem} \label{th:symplectic-form}
The bundle $\tilde{F} \to {\tilde{\nu}_D} \xrightarrow{\tilde \pi} D$ admits closed $2$-form $\o_{\tilde{\nu}_D}$ so that:
\begin{itemize}
\item The restriction of $\o_{\tilde{\nu}_D}$ to each fiber $\tilde{F}_x$ coincides with $\O_{\tilde{F}}$.

\item If $E \subset {\tilde{\nu}_D}$ is the exceptional locus, 
then the form $\o_{\tilde{\nu}_D}$ is non-degenerate on a
neighborhood $U^{E}$ of $E$ in ${\tilde{\nu}_D}$.
\end{itemize}
\end{theorem}

The form $\o_{\tilde{\nu}_D}$ has the expression
\begin{align*}\label{eqn:alfa}
 \o_{\tilde{\nu}_D} & = K \tilde \pi^*(\o_D) + \eta + \textstyle \sum_{\a} d[(\pi^* \rho_\a) \theta_\a] \\
                    & = K \tilde \pi^*(\o_D) + \eta + d \mu \, .                    
\end{align*}
for some $K>0$ large enough, a 
finite atlas of symplectic-bundle charts $\phi_\a: U_{\a} \x \tilde{F} \to V_\a \subset {\tilde{\nu}_D}$,
some $1$-forms $\theta_\a$, and a partition of unity $\rho_\a$ subordinated to the cover $U_\a$ of $D$.

\medskip

Consider the resolution map $b: \tilde \nu_D \to \nu_D$ and $E=b^{-1}(D)$.
Recall that $b$ is a diffeomorphism from $\tilde \nu_D-E$ to $\nu_D-D$.
Consider the almost complex structure $J_{\nu_D}$ from
Proposition \ref{prop:linear-Kahler}, defined in $B_r(D) \subset \nu_D$.
We have an induced almost complex structure $b^*J_{\nu_D}=(b_*)^{-1} J_{\nu_D} b_*$ 
defined in $B_r(E) - E$. In the next proposition we see that the symplectic form
of $\o_{\tilde \nu_D}$ is compatible with $b^* J_{\nu_D}$.


\begin{proposition} \label{prop:positive-outside-E}
Consider the symplectic form $\o_{\tilde \nu_D}$ from Theorem \ref{th:symplectic-form}
defined in some neighborhood $B_r(E)=b^{-1}(B_r(D)) \subset \tilde \nu_D$.
Then 
for $r>0$ small enough, $\o_{\tilde{\nu}_D}$ is $b^*J_{\nu_D}$-positive on $B_r(E)-E$.
\end{proposition}

\begin{proof}
Consider the splitting $T \nu_D= \cHH \oplus \cV=TF^{\perp \O_{\nu_D}} \oplus TF$, and the splitting
$$
T \tilde \nu_D=\tilde \cHH \oplus \tilde \cV = T \tilde F^{\perp \tilde \o_{\nu_D}} \oplus T \tilde F \, .
$$ 
Recall that $F=\CC^k/\G$ has a natural complex structure $J_F$, and $\tilde F$ has also a complex structure $J_{\tilde F}$ 
from Proposition \ref{prop:resolution-fiber-Kahler} so that the resolution map $b|_{\tilde F}: (\tilde F,J_{\tilde F}) \to (F,J_F)$ is holomorphic. Since $J_{\nu_D}|_{F}=J_F$, it follows that $b^*J_{\nu_D}|_{\tilde F}=b^*J_F=J_{\tilde F}$. 
We check now that $\o_{\tilde{\nu}_D}$ is $b^*J_{\nu_D}$-positive in $\tilde \cHH$ and $\tilde \cV$.
If $u \in \tilde \cV= T \tilde F$ we have
\begin{align*}
\o_{\tilde \nu_D}(u,b^*J_{\nu_D} u)& =\o_{\tilde \nu_D}(u,(b_*)^{-1}J_{\nu_D} b_* u)  \\
                                   &= \o_{\tilde \nu_D}(u,J_{\tilde F} u) = \O_{\tilde F} (u,J_{\tilde F}u) >0 \, .
\end{align*}
If $u \in \tilde \cHH$ then
\begin{equation} \label{eq:omega-tilde-nu}
\o_{\tilde \nu_D}(u,b^*J_{\nu_D} u) =
K \tilde \pi^* \o_D(u,b^*J_{\nu_D} u) + \eta(u,b^*J_{\nu_D} u) + d \mu (u,b^*J_{\nu_D} u) \, .
\end{equation}
If we see that $\tilde \pi^* \o_D(b^* J_{\nu_D}u,u)>0$ we are done, because then we can take $K>0$ big enough so that
$K \tilde \pi^* \o_D$ dominates and $\o_{\tilde \nu_D}(u,b^*J_{\nu_D} u)>0$. We compute
\begin{align*}
\tilde \pi^* \o_D(u,b^* J_{\nu_D}u) &=\o_D(\tilde \pi_*u, \tilde \pi_* b_*^{-1} J_{\nu_D} b_*u) \\
                                    &=\pi^* \o_D(b_*u,J_{\nu_D} b_*u) \\
                                    & =\pi^* \o_D(p_{\cHH}( b_*u), J_{\cHH} p_{\cHH}( b_*u)) \ge c | p_{\cHH}( b_*u)|^2,
\end{align*}
where we have used that $\tilde \pi_*=\pi_*b_*$ in the second equality, and that $\pi^* \o_D$
is $J_{\cHH}$-positive in the horizontal distribution $\cHH$ at points of some neighborhood $B_{\d}(D) \subset \nu_D$.
Note that $p_{\cHH} : T \nu_D \to \cHH$ is the projection onto the horizontal distribution.
Also, $u \in \tilde \cHH= T \tilde F^{\perp \o_{\tilde \nu_D}}$ is away from $\tilde \cV=T \tilde F$. 
On the other hand $\ker b_*$ is nonzero only at points of $E$, 
and on such points $\ker b_* \subset T \tilde F$, so $|b_*u| \ge c_1 |u|$ for all vectors $u \in \tilde \cHH$
and all points in some neighborhood $B_{\d}(E)$, with $c_1>0$ a constant.
Moreover, since $b_*|_{\tilde \cHH}$ is injective also at the points of $E$ and 
$b_*(T \tilde F) \subset TF$, it follows that $b_*(\tilde \cHH) \oplus \cV = T \nu_D$,
so $| p_{\cHH}( b_*u)| \ge c_2 |u|$ for all $u \in \tilde \cHH$
and all points in some neighborhood $B_{\d}(E)$, with $c_2>0$ a constant.

It follows that $\tilde \pi^* \o_D(b^* J_{\nu_D}u,u) \ge c_3 |u|^2$ for all vectors $u \in \tilde \cHH$
and all points in $B_{\d}(E) - E$, with $c_3>0$ a constant. Now, looking at \eqref{eq:omega-tilde-nu},
it is immediate that if we take $K>0$ large enough
we can achieve that $\o_{\tilde \nu_D}(u,b^*J_{\nu_D} u)>0$ for all non-zero vectors $u \in \tilde \cHH$
and all points in $B_{\d}(E)-E$, as desired.
\end{proof}

\section{Gluing the symplectic form}

Finally, we glue the symplectic form $\o_{\tilde\nu_D}$ constructed in Theorem \ref{th:symplectic-form}
with the symplectic form of the symplectic orbifold $(X, \o)$. 
Fix a neighbourhood $B_{r_0}(E)\subset\tilde\nu_D$ of the exceptional locus
such that $\o_{\tilde\nu_D}$ is symplectic on $B_{r_0}(E)$, as provided by Theorem \ref{th:symplectic-form},
and so that $\o_{\tilde\nu_D}$ is $b^*J_{\nu_D}$-positive in $B_{r_0}(E)-E$ as in Proposition \ref{prop:positive-outside-E}.

\begin{proposition} \label{prop:gluing}
For $\e>0$ small enough there exists a symplectic form $\O_{\tilde \nu_D}$ on $B_{r_0}(E)$ so that $\O_{\tilde \nu_D}
=\e \o_{\tilde \nu_D}$ on some small neighborhood $B_{\d/4}(E) \subset B_{r_0}(E)$, and
$\O_{\tilde \nu_D}={b}^*(\O_{\nu_D})$ outside some larger neighborhood $B_{2\d}(E)\subset B_{r_0}(E)$.
\end{proposition}

\begin{proof} 

By construction $\o_{\tilde\nu_D}= K \tilde \pi^*(\o_D) + \eta + \sum_{\a} d((\tilde \pi^* \rho_\a) \theta_\a)$, where
the form $\eta$ is a representative of the Poincar\'e dual of the homology class
given by the cycle $A=\sum_\a  \sum_i a_i Q_\a \x Z_i$.
In particular we can take $\eta$ to be very close to a Dirac delta around the cycle $A$, hence
we can suppose that the support of $\eta$ is contained in a small neighborhood of $E$, 
say  $B_{\d/2}(E)$.

On the other hand, consider the orbifold normal bundle $\pi: \nu_D \to D$ of $D$ in $X$ 
with symplectic form $\O_{\nu_D}$ constructed in Proposition \ref{prop:linear-kahler-pro},
so that a neighbourhood of the zero section in $(\nu_D, \O_{\nu_D})$ is symplectomorphic 
to a tubular neighbourhood of $D$ in $X$. Consider also $\bar \O_{\nu_D}$ 
the $J_{\nu_D}$-positive interpolation of $\O_{\nu_D}$ with $0$ 
from Proposition \ref{prop:linear-symplectic-vanish-fiber}. We construct $\bar \O_{\nu_D}$
so that $\bar \O_{\nu_D}=0$ on $B_{\d/4}(D)$, $\bar \O_{\nu_D}=\O_{\nu_D}$ outside $B_{\d/2}(D)$,
and $\bar \O_{\nu_D}$ is $J_{\nu_D}$-positive outside $B_{\d/4}(D)$.
By construction we have $\bar \O_{\nu_D}=\pi^* \o_D -\tfrac{1}{4} d J_{\cV} d (h(|z|^2))$ so it follows that
$$
{b}^*(\bar \O_{\nu_D})= \tilde \pi^* (\o_D) -\tfrac{1}{4} d( b^* (J_{\cV} d (h(|z|^2)) ) \, .
$$ 
On the other hand, outside the support of $\eta$, we have 
$\o_{\tilde\nu_D} =K \tilde \pi^*(\o_D) + d \left(\sum_{\a} (\tilde \pi^* \rho_\a) \theta_\a\right)$.
This implies that $K {b}^*(\bar \O_{\nu_D})$ and $\o_{\tilde\nu_D}$ define the same cohomology class outside
$B_{\d/2}(E)$, so there exists a $1$-form $\g$ such that $ \o_{\tilde\nu_D} - K b^*(\bar \O_{\nu_D}) = d \g$ 
on $B_{r_0}(E) - B_{\delta/2}(E)$. 
Define 
$$
\O_{\tilde \nu_D}= {b}^*(\bar \O_{\nu_D}) + \e\, d (\rho \g),
$$
with $\rho:\tilde\nu_D \to [0,1]$ a bump function
so that $\rho \equiv 1$ on $B_\d(E)$ and $\rho \equiv 0$ outside $B_{2\d}(E)$.
On $B_{\d}(E)  -  B_{\d/2}(E)$ the above formula for $\O_{\tilde \nu_D}$ satisfies 
\begin{equation}\label{eqn:OW}
\O_{\tilde \nu_D}={b}^*( \bar \O_{\nu_D}) + \e \, d  \g=
(1-K \e) {b}^*( \bar \O_{\nu_D})+ \e \o_{\tilde\nu_D}\, ,
\end{equation}
so we extend $\O_{\tilde \nu_D}$ with the same formula to $B_{\d/2}(E)$ and we have a closed $2$-form $\O_{\tilde \nu_D}$
defined in $B_{r_0}(E)$. Let us see that $\O_{\tilde \nu_D}$ is symplectic by cases.
\begin{itemize}
\item On $B_{\d/2}(E)$ the form $\O_{\tilde \nu_D}$ satisfies 
$$
\O_{\tilde \nu_D} = (1-K \e) {b}^*( \bar \O_{\nu_D})+ \e \o_{\tilde\nu_D}\, ,
$$
and in $B_{\d/4}(E)$ the above becomes $\O_{\tilde \nu_D} = \e \o_{\tilde\nu_D}$. Hence $\O_{\tilde \nu_D}$
is clearly symplectic in $B_{\d/4}(E)$. To see that $\O_{\tilde \nu_D}$ is symplectic in $B_{\d/2}(E)  -  B_{\d/4}(E)$,
we note that $\bar \O_{\nu_D}$ is $J_{\nu_D}$-positive outside $B_{\d/4}(D)$ and
$\o_{\tilde\nu_D}$ is $b^*J_{\nu_D}$-positive outside $E=b^{-1}(D)$.
This yields that both $b^* \bar \O_{\nu_D}$ and $\o_{\tilde\nu_D}$ are 
$b^* J_{\nu_D}$-positive forms en $B_{\d/2}(E)  -  B_{\d/4}(E)$.
Hence, if we take $\e<\frac{1}{K}$ we have $\O_{\tilde \nu_D}(u,b^*J_{\nu_D}u) >0$.

\item On $B_{\d}(E)  -  B_{\d/2}(E)$, since $\bar \O_{\nu_D}= \O_{\nu_D}$, the form $\O_{\tilde \nu_D}$ satisfies 
$$
\O_{\tilde \nu_D} = (1-K \e) {b}^*( \O_{\nu_D})+ \e \o_{\tilde\nu_D}
$$
and both $b^* \O_{\nu_D}$ and $\o_{\tilde\nu_D}$ are $b^* J_{\nu_D}$-positive forms,
hence $\O_{\tilde \nu_D}(u,b^*J_{\nu_D}u) >0$.

\item On $B_{2\d}(E)  -  B_{\d}(E)$ we have $\O_{\tilde \nu_D}={b}^*( \O_{\nu_D}) + \e \, d (\rho \g)$.
As ${b}^*( \O_{\nu_D})$ is $b^* J_{\nu_D}$-positive outside $E$, there exist a constant $c>0$
such that 
$$
{b}^*( \O_{\nu_D})(u,b^* J_{\nu_D}u) \ge c |u|^2,
$$ 
for all points of $B_{2\d}(E)  -  B_{\d}(E)$
and all tangent vectors. From here it follows that
\begin{align*}
\O_{\tilde \nu_D} (u,b^* J_{\nu_D}u) & ={b}^*( \O_{\nu_D})(u,b^* J_{\nu_D}u) + \e \, d (\rho \g)(u,b^* J_{\nu_D}u) \\
                                     & \ge c |u|^2 - \e || d (\rho \g) || \cdot || b^* J_{\nu_D} || \cdot |u|^2 \\
                                     & \ge  \tfrac{c}{2} |u|^2,
\end{align*}
if we take $\e<\frac{c}{2 ||d (\rho \g)|| \cdot || b^* J_{\nu_D}||}$ .

\item Finally, on $B_{r_0}(E)  -  B_{2\d}(E)$ we have $\O_{\tilde \nu_D}={b}^*(\O_{\nu_D})$ as we want.
\end{itemize}

\end{proof}

Take the form $\O_{\tilde \nu_D}$ constructed in Proposition \ref{prop:gluing}. 
It is symplectic on some neighborhood $B_{r_0}(E) \subset \tilde\nu_D$ of $E$.
By Proposition \ref{prop:tubular-sympl}, for $r_0>0$ small, $B_{r_0}(D) \subset \nu_D$ 
and some neighborhood $\cV \subset X$ of $D$ are symplectomorphic via
$$
\varphi: (B_{r_0}(D), \O_{\nu_D}) \to (\cV, \o) \, .
$$
We define 
 $$
 \tilde{X}=B_{r_0}(E) \cup_f (X  -  \varphi (B_{2 \d}(D)) ) ,
 $$
with $\d>0$ as given in Proposition \ref{prop:gluing}.
The gluing map is 
$$
f= \varphi \circ b : (B_{r_0}(E)  -  B_{2 \d}(E),\O_{\tilde \nu_D}) \to (V,\o) \subset \cV \subset X,
$$
whose image is some open set $V \subset \cV$. 
Since $f^*(\o)= b^* \varphi^* \o= b^* \O_{\nu_D}= \O_{\tilde \nu_D}$, we see that
$f$ is a symplectomorphism. Hence $\tilde{X}$ is a symplectic manifold. We have proved the following:

\begin{theorem} \label{thm:main}
Let $(X, \o)$ be a symplectic orbifold such that all its isotropy set consists of homogeneous disjoint embedded submanifolds
in the sense of definition \ref{def:HI}.
There exists a symplectic manifold $(\tilde{X}, \tilde{\o})$ and a smooth map $b: (\tilde{X},\tilde{\o}) \to (X,\o)$
which is a symplectomorphism outside an arbitrarily small neighborhood of the isotropy points.
\end{theorem}

\begin{remark}
If the isotropy submanifold $D \subset X$ is such that its normal tangent spaces $F=\CC^k/\G$
are not singular spaces (for instance, when $D$ has codimension $2$ in $X$), 
then the constructive resolution has $\tilde{F}=F$
and $E=D$. In this case Theorem \ref{thm:main} serves to obtain a smooth symplectic form
on $X$ from an orbifold symplectic form. This construction appears in \cite{MRT}.
\end{remark}

\section{Examples}
In this section, we want to give some examples where we can apply Theorem \ref{thm:main-thm}.

\smallskip \noindent \textbf{Example 1. A symplectic divisor.}
Let $(X,\o)$ be a symplectic orbifold of dimension $2n$ such that the isotropy locus $D\subset X$ is a divisor,
that is, $\dim D=2n-2$, and the isotropy is given by $\G=\ZZ_k=\la g\ra$ acting on the normal space $\CC$ by
$g(z)=e^{2\pi i/k} z$. Then $X$ is topologically a manifold since $\CC/\ZZ_k$ is homeomorphic to $\CC$.
The algebraic resolution of $F=\CC/\ZZ_k$ is given by $\tilde F=\CC$, with map $b:\tilde F\to F$, $b(w)=w^k$.
Note that $b$ is the homeomorphism mentioned above. Theorem \ref{thm:main-thm} applies to get a smooth
symplectic manifold $(\tilde X,\o_{\nu_D})$ with a map $b:\tilde X\to X$ which is a symplectomorphism
outside a small neighbourhood of $D$. 

Note that $b$ is bijective, hence a homeomorphism. Then we can
identify $\tilde X\cong X$, and hence Theorem \ref{thm:main-thm} in this case means that we can change
the orbifold atlas of $X$ by a smooth atlas, and the orbifold symplectic form $\omega$ by a smooth symplectic
form $\o_{\nu_D}$.
This process is the reverse process to that of \cite{MRT}, where we started we a smooth symplectic manifold
to produce an orbifold symplectic form with some prescribed isotropy group (in \cite{MRT} the dimension of
the orbifold is $4$, but the result holds for arbitrary dimension).

\smallskip \noindent \textbf{Example 2. A product.}
Let $(M,\o_1)$ be a symplectic orbifold with isolated orbifold singularities. By \cite{CFM}, we have a symplectic
resolution $b:(\tilde M,\tilde\o_1)\to (M,\o_1)$. Let $(N,\o_2)$ be a smooth symplectic manifold. Then
$(X=M\x N, \o_1+\o_2)$ is a symplectic orbifold with homogeneous isotropy sets. Actually, if $x\in M$ is
a singular point of $M$, then $D=\{x\}\x N$ is an isotropy submanifold of $X$. The map $b:
(\tilde M\x N,\tilde\o_1+\o_2) \to (M\x N,\o_1+\o_2)$ is a symplectic resolution,
agreeing with Theorem \ref{thm:main-thm}. In this case, the symplectic normal bundle to $D$ is trivial.

\smallskip \noindent \textbf{Example 3. Symplectic bundle over an orbifold.}
Let $(F,\o_F)$ be a symplectic manifold, $(B,\o_B)$ a symplectic orbifold with isolated singularities,
and let $F\to M\stackrel{\pi}{\longrightarrow} B$
be a smooth bundle, where $(M,\o)$ is a symplectic orbifold such that $(F_x,\o|_{F_x})$ is
symplectomorphic to $(F,\o_F)$, for all fibers $F_x=\pi^{-1}(x)$, $x\in B$ (that is, $M$ is a symplectic
bundle over an orbifold symplectic base). 
For a small orbifold chart $(U,V,\varphi,\G)$ of $B$, we have $\pi^{-1}(V) \cong V\x F \cong (U/\G)\x F=(U\x F)/\G$,
where $\G$ acts on the first factor. As we are assuming that $B$ has isolated singularities,
the isotropy sets are $F_x$, where $x\in B$ is a singularity of $B$. Theorem \ref{thm:main-thm}
guarantees the existence of a symplectic resolution of $M$. 

Actually, the resolution is given as follows. Take a resolution
$b:(\tilde B,\tilde\o_B)\to (B,\o_B)$ provided by \cite{CFM}, and take the pull-back of the bundle
$F\to \tilde M\stackrel{\tilde\pi}{\longrightarrow} \tilde B$. Then for every singular point $x\in B$ with orbifold
chart $(U,V,\varphi,\G)$, we glue the symplectic form $\tilde\o_B \x \o_F$ on $\tilde\pi^{-1}(\tilde V)
\cong \tilde V\x F$ to $\o_M$ along the complement of a neighbourhood of $F_x$.
Theorem \ref{thm:main-thm} does the job without having to care about the details.

\smallskip \noindent \textbf{Example 4. Mapping torus.}
Let $(M,\o_M)$ be a compact symplectic orbifold with isolated singularities. Let $f:M\to M$ be an orbifold symplectomorphism
and consider the mapping torus $M_f=(M\x [0,1])/\sim$ with $(x,0)\sim (f(x),1)$. Let $t$ be the coordinate of
$[0,1]$ and consider a circle $S^1$ with coordinate $\theta$. Then $X=M_f \x S^1$ is a symplectic orbifold
with symplectic form $\o=\o_M+ dt\wedge d\theta$. The isotropy sets are $2$-tori. Take a singular point $x\in M$
and let $x_0=x, x_1=f(x_0), x_2=f^2(x_0),\ldots$ be the orbit of $x$. As all of them are singular points and
there are finitely many of them in $M$,
there is some $n>0$ such that $x_n=x_0$, and we take the minimum of such $n$. Consider the
circle $C_x$ given by the image of $\{x_0,\ldots, x_{n-1}\}\x [0,1]$ in $M_f$, which is a $n:1$ covering of
$[0,1]/\sim=S^1$. 
Then $D=C_x \x S^1$ is an isotropy set of $X=M_f\x S^1$. Theorem \ref{thm:main-thm}
gives a symplectic resolution of $X$. This can be constructed alternatively by taking the symplectic resolution
$b: \tilde M\to M$ of $M$ given by \cite{CFM}. If we arrange to do it in an equivariant way around the singular points, 
then we may lift $f$ to a symplectomorphism
$\tilde f:\tilde M\to \tilde M$ of the resolved manifold, and $\tilde X=\tilde M_{\tilde f}\x S^1$ 
is a symplectic resolution of $X$.

\smallskip \noindent \textbf{Example 5. An example with non-trivial normal bundle.}
Take a standard $6$-torus $T^6=\RR^6/\ZZ^6$ with the standard symplectic form
$\omega=dx_1\wedge dx_2+dx_3\wedge dx_4+dx_5\wedge dx_6$, and consider the maps 
 \begin{align*}
f(x_1,x_2,x_3,x_4,x_5,x_6) &= (x_1,x_2,-x_3,-x_4,-x_5,-x_6), \\ 
g(x_1,x_2,x_3,x_4,x_5,x_6) &= f(x_1+\frac12,x_2,x_3,x_4,-x_5,-x_6) .
 \end{align*}
Then $X=T^6/\la f,g\ra$ is a symplectic orbifold. The isotropy locus are the subsets 
$S_{\mathbf{a}}=\{(x_1,x_2,a_3,a_4,a_5,a_6) | (x_1,x_2)\in \RR^2\}$, for $\mathbf{a}=(a_3,a_4,a_5,a_6) \in
\{0,1/2\}^4$. Each of them is isomorphic to $\RR^2 /\la (1/2,0), (0,1)\ra$. The normal structure is 
$F=\CC^2/\ZZ_2$, with action $(z_1,z_2)\sim (-z_1,-z_2)$. The normal bundle is the quotient of
the trivial bundle $T^2\x F \to T^2$ over $T^2=\RR^2/\ZZ^2$, by the map $g$, hence it is non-trivial
(although it is trivializable).

\smallskip \noindent \textbf{Example 6. Resolving the quotient of a symplectic nilmanifold.}
To give an explicit example of a resolution, we shall take a symplectic $6$-nilmanifold from \cite{BM2}
and perform a suitable quotient to get a symplectic $6$-orbifold with homogeneous isotropy. 
For instance we take the nilmanifold
corresponding to the Lie algebra $L_{6,10}$ of Table 2 in \cite{BM2}, which is symplectic since it appears
in Table 3 of \cite{BM2}. Take the group of $(7\x 7)$-matrices given by the matrices
$$
\left(\begin{array}{ccccccc} 
1 & x_2 & x_1 & x_4 & x_1x_2 & x_5 & x_6 \\
0 & 1 & 0 & -x_1 & x_1 & x_1^2/2 & x_3 \\
0 & 0 & 1 & 0 & x_2 & - x_4 & x_2^2/2 \\
0 & 0 & 0 & 1 & 0 & 0 & 0 \\
0 & 0 & 0 & 0 & 1 & x_1 & x_2 \\
0 & 0 & 0 & 0 & 0 & 1 & 0 \\
0 & 0 & 0& 0 & 0 & 0 & 1 \end{array}\right),
$$
where $x_{i} \in {\RR}$, for any $i =1, \ldots, 6$. Then, a global system of coordinate functions
$\{x_{1}, \ldots, x_{6}\}$ for $G$ is given by $x_{i}(a)=x_{i}$, with $i=1, \ldots, 6$.
Note that if a matrix $A\in G$ has coordinates $a_i$, then the change of coordinates
of $a\in G$ by the left translation $L_{A}$ are given by
\begin{align*}
&L_{A}^*(x_1) = x_1+ a_1, &
&L_{A}^*(x_2) = x_2+ a_2, \\
&L_{A}^*(x_3) = x_3 + a_{1} x_2 + a_{3}, &
& L_{A}^*(x_4) = x_4- a_{2} x_1 + a_{4}, \\
&L_{A}^*(x_5) = x_5+ \frac{1}{2} a_{2} x_{1}^2 - a_{1} x_4 + a_{1} a_{2} x_1 + a_{5}, \hspace{-2cm}\\
&L_{A}^*(x_6) = x_6 + \frac{1}{2} a_{1} x_{2}^2 + a_{2} x_3 + a_{1} a_{2} x_2 + a_{6}. \hspace{-2cm}
\end{align*}

A standard calculation shows that a basis for the left invariant $1$-forms on $G$ consists of
$$
\{dx_{1}, dx_{2}, dx_{3}-x_1dx_2, dx_{4} + x_{2} dx_{1}, dx_{5} + x_{1} dx_{4}, dx_{6} - x_2 dx_{3}\}.
$$
Let $\Gamma$ be the discrete subgroup of $G$ consisting of
matrices with entries $(x_1,x_2,\ldots, x_6) \in (2\ZZ)^2 \x \ZZ^4$, that is
$x_i$ are integer numbers and $x_1,x_2$ are even.
It is easy to see that $\G$ is a subgroup of $G$.
So the quotient space of right cosets $M\,=\,\Gamma{\backslash} G$
is a compact $6$-manifold.
Hence the $1$-forms
\begin{align*}
&e_1=dx_{1}, e_2=-dx_{2}, e_3=dx_{3}-x_1dx_2-dx_4-x_2dx_1=d(x_3-x_4-x_1x_2), \\
&e_4=dx_4+x_2dx_1, e_5= dx_{5} + x_{1} dx_{4}, e_6=dx_{6} - x_2 dx_{3}
\end{align*}
satisfy
$$
de_1=de_2=de_3=0, de_4= e_1e_2, de_5=e_1e_4, de_6=e_2e_3+e_2e_4\, .
$$
This coincides with $L_{6,10}$ in  Table 2 in \cite{BM2}.
The symplectic form of $M$ is $\omega=e_1e_6+e_2e_5-e_3e_4$ (see Table 3 in \cite{BM2}).

Now we consider the map $\varphi(x_1,x_2,x_3,x_4,x_5,x_6)=(x_1,-x_2,-x_3,-x_4,-x_5,x_6)$. This
is given in terms of the matrices as $\varphi(A)=PAP$, where $P$ is the diagonal matrix 
$P=\mathrm{diag}(1,-1,1,-1,-1,-1,1)$. Note
that for $N\in \G$, $PNAP=(PNP)(PAP)$. As $\varphi(\G)=\G$, we see that $\varphi$ descends
to $M=\G \backslash G$. This is clearly a symplectomorphism with $\varphi^2=\Id$, hence
$$
X=M /\la \varphi\ra
$$
is a symplectic orbifold. The isotropy locus is formed by the sets
$$
S_{\mathbf{b}}=\{ (x_1, b_2, b_3, b_4 - b_2x_1 ,b_5+\frac12 b_2 x_1^2,x_6) | (x_1,x_6 )\in \RR^2\},
$$
for $\mathbf{b}= (b_2, b_3,b_4,b_5) \in \{0,1\} \x \{0,1/2\}^3$. This is a collection of $16$ tori, each of them
of homogeneous isotropy $\CC^2/\ZZ_2$.
This is computed computed solving the equation $\varphi(x)= A x$ for some $A \in \G$,
which translates to $x_1=L_A^*(x_1)$, $-x_i=L_A^*(x_i)$ for $2 \le i \le 5$ and $x_6=L_A^*(x_6)$.

The above manifold $M$ is a circle bundle (with coordinate $x_6$) over a mapping torus (with coordinate
$x_1$) of a $4$-torus (with coordinates $x_2,x_3,x_4,x_5$). Then we take a quotient of $T^4$ by
$\ZZ_2$ acting as $\pm \Id$. So this fits with Example 4 above.

Let us compute the Betti numbers of the resolution $\tilde X$ of $X$. The Betti numbers of $M$ appear
in Table 2 of \cite{BM2} and are $b_1(M)=3,b_2(M)=5,b_3(M)=6$. Easily we get that $H^1(M)=\la e_1,e_2,e_3\ra$
and $H^2(M)=\la e_2e_3, e_1e_5, e_1e_3, e_2e_6, e_3e_6+e_4e_6\ra$. Taking the invariant part
by the action of $\varphi$, we have
 $$
 H^1(X)=\la e_1\ra, \quad 
 H^2(X)=\la e_2e_3\ra,
 $$
so $b_1(X)=1$ and $b_2(X)=1$. By Poincar\'e duality, $b_4(X)=b_5(X)=1$. Now $\chi(X)=0$ since
$\chi(M)=0$ and the ramification locus are $T^2$ which have $\chi(T^2)=0$. Therefore $b_3(X)=2$.

The resolution process changes $F=\CC^2/\ZZ_2$ by the single blow-up at the origin $\tilde F$, which 
has exceptional divisor $Z=\CP^1$ with $Z^2=-2$. Then each exceptional locus increases by $1$
the second Betti number $b_2$ (cf.\ the computations of cohomology in \cite{FFKM}). Therefore
$b_1(\tilde X)=1, b_2(\tilde X)=1+16=17$. By Poincar\'e duality, 
$ b_4(\tilde X)=17, b_5(\tilde X)=1$. Again $\chi(\tilde X)=0$, since the exceptional divisors are 
$\CP^1$-bundles over $T^2$ and hence they have $\chi(E)=0$. So $b_3(\tilde X)=34$.

\end{document}